\documentclass[12pt]{amsart}
\usepackage{amscd,amssymb}
\usepackage{amsthm,amsmath,enumerate,amssymb,wasysym}
\usepackage[matrix,arrow]{xy}
\usepackage{enumerate}
\usepackage{longtable}
\usepackage{comment}

\theoremstyle{definition}

\newtheorem*{example*}{Example}
\newtheorem{definition}[equation]{Definition}
\newtheorem{theorem}[equation]{Theorem}
\newtheorem{lemma}[equation]{Lemma}
\newtheorem{corollary}[equation]{Corollary}

\newtheorem*{conjecture*}{Conjecture}

\newtheorem*{question*}{Question}

\newtheorem*{problem*}{Problem}
\newtheorem*{theorem*}{Theorem}

\theoremstyle{remark}
\newtheorem{remark}[equation]{Remark}
\newtheorem*{remark*}{Remark}

\address{\emph{Ivan Cheltsov}
\newline
\textnormal{University of Edinburgh, United Kingdom}
\newline
\textnormal{HSE University, Russian Federation}
\newline
\textnormal{\texttt{I.Cheltsov@ed.ac.uk}}}

\address{\emph{Kewei Zhang}
\newline
\textnormal{Peking University, Beijing, China}
\newline
\textnormal{\texttt{kwzhang@pku.edu.cn}}}

\makeatletter\@addtoreset{equation}{section} \makeatother

\author{Ivan Cheltsov and Kewei Zhang}

\title{Delta invariants of smooth cubic surfaces}

\begin{document}

\begin{abstract}
We prove that $\delta$-invariants of smooth cubic surfaces are at least $\frac{6}{5}$.
\end{abstract}

\maketitle

All varieties are assumed to be projective and defined over $\mathbb{C}$.

\section{Introduction}
\label{section:intro}

The existence of K\"ahler-Einstein metrics on Fano manifolds is an important problem in complex geometry.
By Yau--Tian--Donaldson conjecture (confirmed in \cite{CDS,T15}), we know that all $K$-stable Fano manifolds are K\"ahler-Einstein.
Moreover, we also know explicit criteria that can be used to verify $K$-stability in many cases.
One such criterion has been found by Tian in \cite{TianAlpha} and later generalized by Fujita in \cite{FujitaAlpha}.
It is the following

\begin{theorem}[{\cite{TianAlpha,FujitaAlpha}}]
\label{theorem:Tian}
Let $X$ be a Fano manifold of dimension $n\geqslant 2$. If \mbox{$\alpha(X)\geqslant\frac{n}{n+1}$}, then $X$ is $K$-stable.
\end{theorem}

Here, $\alpha(X)$ is the $\alpha$-invariant defined in \cite{TianAlpha}.
By \cite[Theorem~A.3]{CheltsovThreefolds}, one has
$$
\alpha\big(X\big)=\mathrm{sup}\left\{\lambda\in\mathbb{Q}\ \Bigg|%
\aligned
&\text{the log pair}\ \left(X, \lambda D\right)\ \text{is log canonical}\\
&\text{for every effective $\mathbb{Q}$-divisor}\ D\sim_{\mathbb{Q}} -K_X
\endaligned\Bigg.\right\}.
$$
In \cite{Cheltsov}, the first author computed  the $\alpha$-invariants
of two-dimensional Fano manifolds, known as del Pezzo surfaces.
Namely, if $S$ be a smooth del Pezzo surface, then
$$
\alpha(S)=\left\{%
\aligned
&\frac{1}{3}\ \mathrm{if}\ S\cong\mathbb{F}_{1}\ \mathrm{or}\ K_{S}^{2}\in\{7,9\},\\%
&\frac{1}{2}\ \mathrm{if}\ S\cong\mathbb{P}^{1}\times\mathbb{P}^{1}\ \mathrm{or}\ K_{S}^{2}\in\{5,6\},\\%
&\frac{2}{3}\ \mathrm{if}\ K_{S}^{2}=4,\\%
&\frac{2}{3}\ \mathrm{if}\ S\ \mathrm{is\ a\ cubic\ surface\ in}\ \mathbb{P}^{3}\ \mathrm{with\ an\ Eckardt\ point},\\%
&\frac{3}{4}\ \mathrm{if}\ S\ \mathrm{is\ a\ cubic\ surface\ in}\ \mathbb{P}^{3}\ \mathrm{without\ Eckardt\ points},\\%
&\frac{3}{4}\ \mathrm{if}\ K_{S}^{2}=2\ \mathrm{and}\ |-K_{S}|\ \mathrm{has\ a\ tacnodal\ curve},\\%
&\frac{5}{6}\ \mathrm{if}\ K_{S}^{2}=2\ \mathrm{and}\ |-K_{S}|\ \mathrm{has\ no\ tacnodal\ curves},\\%
&\frac{5}{6}\ \mathrm{if}\ K_{S}^{2}=1\ \mathrm{and}\ |-K_{S}|\ \mathrm{has\ a\ cuspidal\ curve},\\%
&1\ \mathrm{if}\ K_{S}^{2}=1\ \mathrm{and}\ |-K_{S}|\ \mathrm{has\ no\ cuspidal\ curves}.\\%
\endaligned\right.%
$$
In particular, if $K_S^2\leqslant 4$, then $S$ is $K$-stable by Theorem~\ref{theorem:Tian}, so that it is K\"ahler-Einstein.
If $K_S^2=5$, then $S$ is unique and $\mathrm{Aut}(S)\cong\mathfrak{S}_5$.
In this case, we have $\alpha_{\mathfrak{S}_5}(S)=2$~by~\cite{Cheltsov},
where $\alpha_{\mathfrak{S}_5}(S)$ is a $\mathfrak{S}_5$-invariant $\alpha$-invariant,
which can be defined similarly to $\alpha(S)$.
Now~using an $\mathfrak{S}_5$-equivariant counterpart of Theorem~\ref{theorem:Tian} in \cite{TianAlpha}, we conclude that the surface $S$ is also K\"ahler-Einstein.
All remaining del Pezzo surfaces are toric, so that they are K\"ahler-Einstein if and only if their Futaki characters vanish \cite{WangZhu}.
Together with Matsushima's obstruction, this give Tian's celebrated

\begin{theorem}[\cite{TianSurface}]
\label{theorem:Tian-surfaces}
A smooth del Pezzo surface admits a K\"ahler-Einstein metric if and only if it is not a blow up of $\mathbb{P}^2$ at one or two points.
\end{theorem}

Note that smooth cubic surfaces form the hardest case in Tian's original proof of this result,
which requires Cheeger--Gromov theory, H\"ormander $L^2$ estimates, partial $C^0$ estimates and the lower semi-continuity of log canonical thresholds.
In this paper, we will give another proof of Theorem~\ref{theorem:Tian-surfaces} in this case using a new criterion for $K$-stability,
which has been recently discovered by Fujita and Odaka in \cite{FujitaOdaka}.
They stated it in terms of the so-called $\delta$-invariant, which we describe now.

Fix a Fano manifold $X$.
For a sufficiently large and sufficiently divisible integer $k$, consider a basis $s_1,\cdots, s_{d_k}$ of the vector space $H^0(\mathcal{O}_X(-kK_X))$, where $d_k=h^0(\mathcal{O}_X(-kK_X))$.
For~this basis, consider the $\mathbb{Q}$-divisor
$$
\frac{1}{kd_k}\sum_{i=1}^{d_k}\big\{s_i=0\big\}\sim_{\mathbb{Q}}-K_X.
$$
Any $\mathbb{Q}$-divisor obtained in this way is called a $k$-basis type (anticanonical) divisor. Let
$$
\delta_k\big(X\big)=\mathrm{sup}\left\{\lambda\in\mathbb{Q}\ \Bigg|%
\aligned
&\text{the log pair}\ \left(X, \lambda D\right)\ \text{is log canonical}\\
&\text{for every $k$-basis type $\mathbb{Q}$-divisor}\ D\sim_{\mathbb{Q}} -K_X
\endaligned\Bigg.\right\}.
$$
Then let
$$
\delta(X)=\limsup_{k\in\mathbb{N}}\delta_k(X).
$$
By \cite[Theorem A]{BJ17}, one has
$$
\frac{\mathrm{dim}(X)+1}{\mathrm{dim}(X)}\alpha\big(X\big)\leqslant\delta\big(X\big)\leqslant\big(\mathrm{dim}(X)+1\big)\alpha\big(X\big).
$$
The number $\delta(X)$ is also referred to as the \textit{stability threshold} (cf. \cite{BJ17,BBJ18}),
because of

\begin{theorem}[{\cite[Theorem~B]{BJ17}}]
\label{theorem:delta}
The following assertions hold:
\begin{enumerate}
\item $X$ is $K$-semistable if and only if $\delta(X)\geqslant1$;
\item $X$ is uniformly $K$-stable if and only if $\delta(X)>1$.
\end{enumerate}
\end{theorem}

How to compute or at least estimate $\delta(X)$ effectively?
In general this is not very easy.
In~\cite{ParkWon}, Park and Won estimated the $\delta$-invariants of all smooth del Pezzo surfaces,
which gave another proof of Tian's Theorem~\ref{theorem:Tian-surfaces}.
But it seems unclear to us how to generalize their approach for higher-dimensional Fano manifolds.
Motivated by this, in our recent joint work with Yanir Rubinstein \cite{CRZ},
we developed new geometric tools to estimate $\delta$-invariants of (log) del Pezzo surfaces,
which enabled us to partially prove a conjecture proposed in \cite{CR15}.
In~this paper, we will use the same methods to give a sharper estimate for the $\delta$-invaraints of smooth cubic surfaces.
To be precise, we prove

\begin{theorem}
\label{theorem:main}
Let $S$ be a smooth cubic surface in $\mathbb{P}^3$. Then $\delta(S)\geqslant\frac{6}{5}$.
\end{theorem}

\begin{corollary}[{\cite{TianSurface,ParkWon}}]
\label{corollary:main}
All smooth cubic surfaces in $\mathbb{P}^3$ are uniformly $K$-stable,
so that they are K\"ahler-Einstein.
\end{corollary}

For a smooth cubic surface $S$, it follows from \cite[Theorem~4.9]{ParkWon} that
$$
\delta(S)\geqslant\frac{36}{31}.
$$
Our bound  $\delta(S)\geqslant\frac{6}{5}$ is slightly better $\smiley$.
Moreover, the~proof of  Theorem~\ref{theorem:main} is completely different from the proof of \cite[Theorem~4.9]{ParkWon}.
The essential ingredient in our proof is a \emph{vanishing order estimate} for basis type divisors (see Theorem~\ref{theorem:Fujita}).
This estimate combined with the techniques from \cite{Cheltsov} give us the desired lower bound for $\delta(S)$.

This paper is organized as follows.
In Section~\ref{section:basic-tool}, we present known results about divisors on smooth surfaces,
and, as an illustration, we give a new proof of \cite[Theorem~4.7]{ParkWon}.
In Section~\ref{section:mult-estimates}, we give various multiplicity estimates for basis type divisors on smooth cubic surfaces,
which will be important to bound their $\delta$-invariants in the proof of Theorem~\ref{theorem:main}.
These estimates also imply that $\delta$-invariants of smooth cubic surfaces are at least~$\frac{18}{17}$.
In~Section~\ref{section:delta-bigger-than-one}, we prove Theorem \ref{theorem:main}.

\smallskip
\textbf{Acknowledgments.}
The authors want to thank Yanir Rubinstein for many helpful discussions.
Ivan Cheltsov was partially supported by the Russian Academic Excellence Project ``5-100''.
Kewei Zhang was supported by the China post-doctoral grant BX20190014.
This paper was finished during the authors' visit to the Department of Mathematics at the University of Maryland, College Park.
The authors appreciate its excellent environment and hospitality.

\section{Basic tools}
\label{section:basic-tool}

In this section, we collect some basic notions and tools that will be used throughout this article.
Let $S$ be a smooth surface, and let $P$ be a point in $S$.
Let $D$ be an effective divisor on $S$.
Suppose that $f=0$ is the local defining equation of $D$ near the point $P$, then the multiplicity of $D$ at $P$,
is defined to be the vanishing order of $f$ at $P$, which we denote by $\mathrm{mult}_P(D)$.
Let $\pi\colon\widetilde{S}\to S$ be the blow up of the point $P$, and let $E$ be the exceptional curve of $\pi$.
Denote by $\widetilde{D}$ the proper transform of $D$ via $\pi$. Then we have
$$
\pi^*(D)=\widetilde{D}+\operatorname{mult}_P(D)\cdot E.
$$

\begin{definition}
\label{definition:local-intersection-number}
Let $C_1$ and $C_2$ be two irreducible curves on a surface $S$. Suppose that $C_1$ and $C_2$ intersect at $P$. Let $\mathcal{O}_{P}$ be the local ring of germs of holomorphic
functions defined in some neighborhood of $P$. Then the local intersection number of $C_1$ and $C_2$ at the point $P$ is defined by
$$
\big(C_1\cdot C_2\big)_P=\dim_{\mathbb{C}} \mathcal{O}_{P}/\langle f_1,f_2\rangle,
$$
where
$f_1=0$ and $f_2=0$ are local defining functions of $C_1$ and $C_2$ around the point $P$.
The global intersection number $C_1\cdot C_2$ is defined by
$$
C_1\cdot C_2=\sum_{P\in C_1\cap C_2}\big(C_1\cdot C_2\big)_P.
$$
\end{definition}

This definition and the definition of $\mathrm{mult}_P(D)$ extends to $\mathbb{R}$-divisors by linearity.
For instance, say we have a curve $C$ and a $\mathbb{R}$-divisor $\Delta=\sum_ia_iZ_i$, where $Z_i$'s are distinct prime divisors
and $a_i\in\mathbb{R}$. Then
$$
\big(C\cdot\Delta\big)_P=\sum_ia_i\big(C\cdot Z_i\big)_P,
$$
where $(C.Z_i)_P=0$ if $Z_i$ does not pass through the point $P$.

In the following, let $D$ be an effective $\mathbb{R}$-divisor on $S$.
We will investigate how to express the singularity of the log pair $(S,D)$ at the point $P$ in terms of $\operatorname{mult}_P(\cdot)$ and $\big(\cdot\big)_P$.

\begin{lemma}[\cite{KollarNotes}]
\label{lemma:not-lc-mult-large}
If $(S,D)$ is not log canonical at $P$, then $\mathrm{mult}_P(D)>1$.
\end{lemma}

Let $C$ be an irreducible curve on $S$.
Write
$$
D=aC+\Delta,
$$
where $a$ is a non-negative real number that is also denoted as $\mathrm{ord}_C(D)$, and $\Delta$ is an effective $\mathbb{R}$-divisor on $S$ whose support does not contain the curve $C$.

\begin{lemma}[{\cite[Proposition~3.3]{CRZ}}]
\label{lemma:Kewei-inequality}
Suppose that $a\leqslant 1$, the curve $C$ is smooth at~the point $P$, and $\mathrm{mult}_P(\Delta)\leqslant 1$. If $(S,D)$ is not log canonical at~$P$, then
$$
\big(C\cdot\Delta\big)_P>2-a.
$$
\end{lemma}

\begin{corollary}
\label{corollary:inversion-of-adjunction}
If $a\leqslant 1$, the curve $C$ is smooth at $P$, and the log pair $(S,D)$ is not log canonical at $P$, then
$$
\big(C\cdot\Delta\big)_P>1.
$$
\end{corollary}

Let $\pi\colon\widetilde{S}\to S$ be the blow up of the point $P$, and let $E_1$ be the exceptional curve of $\pi$.
Denote by $\widetilde{D}$ the proper transform of $D$ via $\pi$. Then
$$
K_{\widetilde{S}}+\widetilde{D}+\big(\mathrm{mult}_P(D)-1\big)E_1\sim_{\mathbb{R}}\pi^*\big(K_S+D\big).
$$
This implies

\begin{corollary}
\label{corollary:lc-invariant-under-blowup}
The log pair $(S,D)$ is log canonical at $P$ if and only if the log pair
$$
\Big(\widetilde{S},\widetilde{D}+\big(\mathrm{mult}_P(D)-1\big)E_1\Big)
$$
is log canonical along the curve $E_1$.
\end{corollary}

Thus, using Lemma \ref{lemma:not-lc-mult-large} and Corollary~\ref{corollary:lc-invariant-under-blowup}, we obtain the following simple criterion.

\begin{corollary}
\label{corollary:mult-at-Q-lc-at-P}
Suppose that
$$
\mathrm{mult}_Q\big(\pi^*(D)\big)=\mathrm{mult}_P\big(D\big)+\mathrm{mult}_Q\big(\widetilde{D}\big)\leqslant 2
$$
for every point $Q\in E_1$. Then $(S,D)$ is log canonical at $P$.
\end{corollary}

If $D$ is a Cartier divisor, then its volume is the number
$$
\mathrm{vol}(D)=\limsup_{k\in\mathbb{N}}\frac{h^0(\mathcal{O}_S(kD)}{k^2/2!},
$$
where the $\limsup$ can be replaced by a limit (see \cite[Example~11.4.7]{Laz-positivity-in-AG}).
Likewise, if $D$ is a $\mathbb{Q}$-divisor, we can define its volume using the identity
$$
\mathrm{vol}(D)=\frac{\mathrm{vol}\big(\lambda D\big)}{\lambda^2}
$$
for an appropriate $\lambda\in\mathbb{Q}_{>0}$.
Then the volume $\mathrm{vol}(D)$ only depends on the numerical equivalence class of the divisor $D$.
Moreover, the volume function can be extended by continuity to $\mathbb{R}$-divisors.
Furthermore, it is log-concave:
\begin{equation}
\label{equation:log-concave}
\sqrt{\mathrm{vol}(D_1+D_2)}\geqslant\sqrt{\mathrm{vol}(D_1)}+\sqrt{\mathrm{vol}(D_2)}.
\end{equation}
for any pseudoeffective $\mathbb{R}$-divisors $D_1$ and $D_2$ on the surface $S$.
For more details about volumes of $\mathbb{R}$-divisors, we refer the reader to \cite{LM-OkounkovBodyTheory,Laz-positivity-in-AG}.

If $D$ is not pseudoeffective, then $\mathrm{vol}(D)=0$.
If the divisor $D$ is nef, then
$$
\mathrm{vol}(D)=D^2.
$$
This follows from the asymptotic Riemann--Roch theorem \cite{Laz-positivity-in-AG}.
If the divisor $D$ is not nef, its volume can be computed using its Zariski decomposition \cite{Fujita,Prokhorov}.
Namely, if $D$ is pseudoeffective, then there exists a nef $\mathbb{R}$-divisor $N$ on the surface $S$ such that
$$
D\sim_{\mathbb{R}} N+\sum_{i=1}^ra_iC_i,
$$
where each $C_i$ is an irreducible curve on $S$ with $N\cdot C_i=0$,
each $a_i$ is a non-negative real number, and the intersection form of the curves $C_1,\ldots,C_r$ is negative definite.
Such decomposition is unique, and it follows from \cite[Corollary~3.2]{BauerKuronyaSzemberg} that
$$
\mathrm{vol}\big(D\big)=\mathrm{vol}\big(N\big)=N^2.
$$
This immediately gives

\begin{corollary}
\label{corollary:vol-replacement-of-line-bundle}
Let $Z_1,\ldots,Z_s$ be irreducible curves on $S$ such that $D\cdot Z_i\leqslant 0$ for every~$i$,
and the intersection form of the curves $Z_1,\ldots,Z_s$ is negative definite.
Then
$$
\mathrm{vol}(D)=\mathrm{vol}\Big(D-\sum_{i=1}^s b_iZ_i\Big),
$$
where $b_1,\ldots,b_s$ are (uniquely defined) non-negative real numbers such that
$$
\Big(D-\sum_{i=1}^s b_iZ_i\Big)\cdot Z_j=0
$$
for every $j$.
\end{corollary}

\begin{corollary}
\label{corollary:vol-replacement-of-line-bundle-simple}
Let $Z$ be an irreducible curve on $S$ such that $Z^2<0$ and $D\cdot Z\leqslant 0$.
Then
$$
\mathrm{vol}(D)=\mathrm{vol}\Big(D-\frac{D\cdot Z}{Z^2} Z\Big).
$$
\end{corollary}

Let $\eta\colon\widehat{S}\to S$ be a birational morphism (possibly an identity) such that $\widehat{S}$ is smooth.
Fix a (non necessarily $\eta$-exceptional) irreducible curve $F$ in the surface $\widehat{S}$. Let
$$
\tau(F)=\sup\Big\{x\in\mathbb{R}_{>0}\ \Big|\ \eta^*(D)-xF\ \text{is numerically equivalent to an effective divisor}\Big\}.
$$
This is called the \textit{pseudo-effective threshold} of $F$.

\begin{theorem}
\label{theorem:Fujita}
Suppose that $S$ is smooth del Pezzo surface, and $D$ is a $k$-basis type divisor with $k\gg 1$.
Then
$$
\text{ord}_F\big(\eta^*(D)\big)\leqslant\frac{1}{(-K_S)^2}\int_0^{\tau(F)}\mathrm{vol}\big(\eta^*(-K_S)-xF\big)dx+\epsilon_k,
$$
where $\epsilon_k$ is a small constant depending on $k$ such that $\epsilon_k\to 0$ as $k\to \infty$.
\end{theorem}

\begin{proof}
This is a very special case of \cite[Lemma 2.2]{FujitaOdaka}.
\end{proof}

In \cite{BJ17,BBJ18}, the quantity
$$
S(F)=\frac{1}{(-K_S)^2}\int_0^{\tau(F)}\mathrm{vol}\big(\eta^*(-K_S)-xF\big)dx
$$
is also called the \emph{expected vanishing order} of anticanonical sections along the divisor $F$.

Theorem \ref{theorem:Fujita} plays a crucial role in the proof of Theorem~\ref{theorem:main}.
As a warm up, let us show how to use Theorem~\ref{theorem:Fujita} to estimate $\delta$-invariants of smooth del Pezzo surfaces of degree~$1$.

\begin{theorem}[{\cite[Theorem~4.7]{ParkWon}}]
\label{theorem:dP-1}
Let $S$ be a smooth del Pezzo surface of degree~$1$. Then $\delta(S)\geqslant\frac{3}{2}$.
\end{theorem}

\begin{proof}
Fix some rational number $\lambda<\frac{3}{2}$.
Let $D$ be a $k$-basis type divisor with $k\gg 1$, and let $P$ be a point in $S$.
We have to show that the log pair $(S,\lambda D)$ is log canonical at~$P$.
By Lemma~\ref{lemma:not-lc-mult-large}, it is enough to prove that
$$
\mathrm{mult}_P\big(D\big)\leqslant\frac{1}{\lambda}.
$$
Applying Theorem~\ref{theorem:Fujita} with $\widehat{S}=\widetilde{S}$, $\eta=\pi$ and $F=E_1$, we see that
$$
\mathrm{mult}_P\big(D\big)\leqslant\int_0^{\tau(E_1)}\mathrm{vol}\big(\pi^*(-K_S)-xE_1\big)dx+\epsilon_k,
$$
where $\epsilon_k$ is a constant depending on $k$ such that $\epsilon_k\to 0$ as $k\to \infty$.

Let us compute~$\tau(E_1)$.
To do this, take a curve $C\in |-K_S|$ such that $P\in C$.
Denote by $\widetilde{C}$ its proper transform on the surface $\widetilde{S}$.
If $C$ is smooth at $P$, then
$$
\pi^*\big(-K_S\big)\sim_{\mathbb{Q}} \widetilde{C}+E_1\ \text{and}\
\widetilde{C}^2=C^2-1=0,
$$
which implies that $\tau(E_1)=1$. In this case, we have
\begin{multline*}
\mathrm{mult}_P\big(D\big)\leqslant\int_0^{1}\mathrm{vol}(\eta^*(-K_S)-xE_2)dx+\epsilon_k=\\
=\int_0^{1}\big((\pi^*(-K_S)-xE_1)\big)^2dx+\epsilon_k=\int_0^{1}\big(1-x^2\big)^2dx+\epsilon_k=\frac{2}{3}+\epsilon_k.
\end{multline*}
Therefore, if $C$ is smooth at $P$, then the log pair $(S,\lambda D)$ is log canonical at $P$ for $k\gg 1$.

To complete the proof, we may assume that $C$ is singular at $P$.
Then $P$ is either nodal or cuspidial, so we have $\operatorname{mult}_PC=2$ and
$$
\pi^*\big(-K_S\big)\sim \widetilde{C}+2E_1,
$$
so that $\tau(E_1)=2$, since $\widetilde{C}^2=-3$. Using Corollary~\ref{corollary:vol-replacement-of-line-bundle-simple}, we see that
$$
\mathrm{vol}\big(\pi^*(-K_S)-xE_1\big)=
\begin{cases}
1-x^2, &0\leqslant x\leqslant \frac{1}{2},\\
\frac{(x-2)^2}{3}, &\frac{1}{2}\leqslant x\leqslant2.\\
\end{cases}
$$
so that $\mathrm{mult}_P(D)\leqslant\frac{5}{6}+\epsilon_k$.
This gives $\delta(S)\geqslant\frac{6}{5}$.
To get $\delta(S)\geqslant\frac{3}{2}$, we must work harder.

Fix a point $Q\in E_1$.
By Corollary~\ref{corollary:mult-at-Q-lc-at-P},
to prove that $(S,\lambda D)$ is log canonical at $P$,
it is enough to show that
$$
\mathrm{mult}_Q\big(\pi^*(D)\big)=\mathrm{mult}_P\big(D\big)+\mathrm{mult}_Q\big(\widetilde{D}\big)\leqslant\frac{2}{\lambda}.
$$
Let $\sigma\colon\widehat{S}\to\widetilde{S}$ be the blow up of the point $Q$.
Denote by $E_2$ the exceptional curve of $\sigma$. Let $\eta=\pi\circ\sigma$.
Applying Theorem~\ref{theorem:Fujita} with $F=E_1$, we see that
$$
\mathrm{mult}_Q\big(\pi^*(D)\big)\leqslant\int_0^{\tau(E_2)}\mathrm{vol}\big(\eta^*(-K_S)-xE_2\big)dx+\varepsilon_k.
$$
Here, as above, the term $\varepsilon_k$ is a constant that depends on $k$ such that $\varepsilon_k\to  0$ as $k\to \infty$.

Let $\widehat{C}$ and $\widehat{E}_1$ be the proper transforms on $\widehat{S}$ of the curves $C$ and $E_1$, respectively.
Then the intersection form of the curves $\widehat{C}$ and $\widehat{E}_1$ is negative definite.
If $Q\in\widetilde{C}$, then
$$
\eta^*(-K_S)\sim_\mathbb{Q}\widehat{C}+2\widehat{E}_1+3E_2,
$$
so that $\tau(E_2)=3$. In this case, using Corollary \ref{corollary:vol-replacement-of-line-bundle-simple}, we see that
$$
\mathrm{vol}\big(\eta^*(-K_S)-xE_2\big)=\mathrm{vol}\Big(\eta^*(-K_S)-xE_2-\frac{x}{2}\widehat{E}_1\Big)=\Big(\eta^*(-K_S)-xE_2-\frac{x}{2}\widehat{E}_1\Big)^2=1-\frac{x^2}{2}
$$
provided that $0\leqslant x\leqslant\frac{2}{3}$. Likewise, if $\frac{2}{3}\leqslant x\leqslant 3$, then Corollary \ref{corollary:vol-replacement-of-line-bundle} gives
\begin{multline*}
\mathrm{vol}\big(\eta^*(-K_S)-xE_2\big)=\mathrm{vol}\Big(\eta^*(-K_S)-xE_2-\frac{5x-1}{7}\widehat{E}_1-\frac{3x-2}{7}\widehat{C}\Big)=\\
=\Big(\eta^*(-K_S)-xE_2-\frac{5x-1}{7}\widehat{E}_1-\frac{3x-2}{7}\widehat{C}\Big)^2=\\
\big(\eta^*(-K_S)-xE_2\big)\Big(\eta^*(-K_S)-xE_2-\frac{5x-1}{7}\widehat{E}_1-\frac{3x-2}{7}\widehat{C}\Big)=\frac{(3-x)^2}{7}.
\end{multline*}
Thus, if $Q\in\widetilde{C}$, then
$$
\mathrm{vol}\big(\eta^*(-K_S)-xE_2\big)=
\begin{cases}
1-\frac{x^2}{2}, &0\leqslant x\leqslant\frac{2}{3},\\
\frac{(3-x)^2}{7}, &\frac{2}{3}\leqslant x\leqslant 3,\\
\end{cases}
$$
so that $\mathrm{mult}_Q(\pi^*(D))\leqslant\frac{2}{\lambda}$ for $k\gg 1$, because
$$
\int_0^{3}\mathrm{vol}(\eta^*(-K_S)-xE_2)dx=\frac{11}{9}<\frac{2}{\lambda}.
$$
Likewise, if $Q\notin\widetilde{C}$, then
$$
\eta^*(-K_S)\sim_\mathbb{Q}\widehat{C}+2\widehat{E}_1+2E_2.
$$
so that $\tau(E_2)=2$.
In this case, using Corollary~\ref{corollary:vol-replacement-of-line-bundle}, we deduce that
$$
\mathrm{vol}\big(\eta^*(-K_S)-xE_2\big)=
\begin{cases}
1-\frac{x^2}{2}, &0\leqslant x\leqslant1,\\
\frac{(2-x)^2}{2}, &1\leqslant x\leqslant 2,\\
\end{cases}
$$
which implies that
$$
\int_0^{2}\mathrm{vol}\big(\eta^*(-K_S)-xE_2\big)dx=1,
$$
so that $\mathrm{mult}_Q(\pi^*(D))\leqslant\frac{2}{\lambda}$ for $k\gg 1$.
\end{proof}

\begin{remark}
In the proof of Theorem~\ref{theorem:dP-1}, there is another way to treat the case when the curve $C$ is singular at $P$,
which relies on Lemma~\ref{lemma:Kewei-inequality}.
Indeed, let $S$ be a smooth del Pezzo surface of degree~$1$, let $P$ be a point in $S$,
and let $C$ be a curve in $|-K_S|$ that passes trough $P$.
Suppose that
$$
\mathrm{mult}_P\big(C\big)=2.
$$
Let $D$ be any $k$-basis type divisor such that $D\sim-K_S$ with $k\gg 1$,
and let $\lambda$ be a positive real number such that $\lambda<\frac{3}{2}$.
Let us show that $(S,\lambda D)$ is log canonical at $P$.
We argue by contradiction. Suppose that $(S,\lambda D)$ is not log canonical at $P$.
Write
$$
D=aC+\Delta,
$$
where $a\geqslant 0$ and $\Delta$ is an effective $\mathbb{Q}$-divisor whose support does not contain $C$.
Note that 
$$
a\leqslant\int_0^\infty\mathrm{vol}(-K_S-xC)dx+\epsilon_k=\frac{1}{3}+\epsilon_k,
$$
where $\epsilon_k$ is a constant that depends on $k$ such that $\varepsilon_k\to  0$ as $k\to \infty$.
Let $m=\mathrm{mult}_P(\Delta)$. Then
$$
1=D\cdot C=(aC+\Delta)\cdot C\geqslant a+2m,
$$
so that $m\leqslant\frac{1-a}{2}$.
Let $\pi\colon\widetilde{S}\rightarrow S$ be the blow up of the point $P$.
Let $E$ be the exceptional curve of $\pi$, and let $\widetilde{C}$ and $\widetilde{\Delta}$
be the proper transforms of $C$ and $\Delta$ on $\widetilde{S}$, respectively.
Then the log pair
$$
\Big(\widetilde{S},\lambda a\widetilde{C}+\lambda\widetilde{\Delta}+\big(\lambda(2a+m)-1\big)E\Big)
$$
is not log canonical at some point $Q\in E$. Note that $\lambda(2a+m)-1<1$. But
$$
E\cdot(\lambda\Delta)=\lambda m\leqslant\lambda\frac{1-a}{2}<\frac{3}{2}\cdot\frac{1}{2}<1.
$$
Thus, we have $Q\in E\cap\widetilde{C}$ by Corollary \ref{corollary:inversion-of-adjunction}.
On the other hand, for $k\gg 1$, we have
$$
\mathrm{mult}_Q\big(\lambda\widetilde{\Delta}+(\lambda(2a+m)-1)E\big)\leqslant2\lambda(a+m)-1\leqslant\lambda\cdot(1+\frac{1}{3}+\epsilon_k)-1\leqslant 1,
$$
so that we can apply Lemma \ref{lemma:Kewei-inequality} to our pair at $Q$.
This gives 
$$
\lambda C\cdot\Delta-2m\lambda+2\lambda(2a+m)-2=\widetilde{C}\cdot\big(\lambda\widetilde{\Delta}+(\lambda(2a+m)-1)E\big)>2-\lambda a,
$$
so that $\lambda(1+4a)>4$, and hence
$$
\frac{3}{2}(1+4\cdot\frac{1}{3}+\epsilon_k)>4,
$$
which is absurd for $\epsilon_k\ll 1$. 
This proves the desired log canonicity of our pair $(S,\lambda D)$.
\end{remark}

The following (simple) result can be very handy.

\begin{lemma}
\label{lemma:simple-inequality}
In the assumptions and notations of Theorem~\ref{theorem:Fujita}, one has
$$
\int_\mu^{\tau(F)}\mathrm{vol}\big(\eta^*(-K_S)-xF\big)dx\leqslant\big(\tau(F)-\mu\big)\mathrm{vol}\big(\eta^*(-K_S)-\mu F\big)
$$
for any $\mu\in[0,\tau(F)]$.
\end{lemma}

\begin{proof}
The assertion follows from the fact that $\mathrm{vol}(\eta^*(-K_S)-xF)$ is a non-increasing function on $x\in[0,\tau(F)]$.
\end{proof}

Using \eqref{equation:log-concave}, this result can be improved as follows:

\begin{lemma}
\label{lemma:barycenter-inequality}
In the assumptions and notations of Theorem~\ref{theorem:Fujita}, one has
$$
\int_\mu^{\tau(F)}\mathrm{vol}\big(\eta^*(-K_S)-xF\big)dx\leqslant\frac{2}{3}\Big(\tau(F)-\mu\Big)\mathrm{vol}\big(\eta^*(-K_S)-\mu F\big)
$$
for any $\mu\in[0,\tau(F)]$.
\end{lemma}

\begin{proof}
The required assertion follows from the proof of \cite[Proposition 2.1]{FujitaBarycenter}.
\end{proof}

We will apply both Lemmas~\ref{lemma:simple-inequality} and \ref{lemma:barycenter-inequality} to estimate the integral in Theorem~\ref{theorem:Fujita}
in the cases when it is not easy to compute.

\section{Multiplicity estimates}
\label{section:mult-estimates}

Let $S$ be a smooth cubic surface in $\mathbb{P}^3$, and let $D$ be a $k$-basis type divisor with $k\gg 1$.
The goal of this section is to bound multiplicities of the divisor $D$ using Theorem~\ref{theorem:Fujita}.
As~in Theorem~\ref{theorem:Fujita}, we denote by $\epsilon_k$ a small number such that $\epsilon_k\to 0$ as $k\to \infty$.

\begin{lemma}
\label{lemma:mult-of-a-line}
Let $L$ be a line on $S$. Then
$$
\text{ord}_L(D)\leqslant\frac{5}{9}+\epsilon_k.
$$
\end{lemma}

\begin{proof}
Let us use assumptions and notations of Theorem~\ref{theorem:Fujita} with $\eta=\mathrm{Id}_S$ and $F=L$.
Let $H$ be a general hyperplane section of the surface $S$ that contains $L$.
Then $H=L+C$, where $C$ is an irreducible conic.
Since $C^2=0$, we have $\tau(F)=1$, so that
$$
\text{ord}_L(D)\leqslant\frac{1}{3}\int_0^1\mathrm{vol}(-K_S-xL)dx+\epsilon_k=\frac{1}{3}\int_0^1(-K_S-xL)^2dx+\epsilon_k=\frac{5}{9}+\epsilon_k
$$
by Theorem~\ref{theorem:Fujita}.
\end{proof}

Fix a point $P\in S$.
Let $\pi\colon\widetilde{S}\to S$ be the blowup of this point.
Denote by $E_1$ the exceptional divisor of $\pi$.
Fix a point $Q\in E_1$.
Let $\sigma\colon\widehat{S}\to\widetilde{S}$ be the blowup of this point.
Denote by $E_2$ the exceptional curve of $\sigma$.
Let $\eta=\pi\circ\sigma$, then
$$
\tau(E_2)=\sup\Big\{x\in\mathbb{R}_{>0}\ \Big|\ \eta^*(-K_S)-xE_2\ \text{is numerically equivalent to an effective divisor}\Big\}.
$$
Applying Theorem~\ref{theorem:Fujita}, we get
\begin{equation}
\label{equation:bound}
\mathrm{mult}_Q\big(\pi^*(D)\big)\leqslant\frac{1}{3}\int_0^{\tau(E_2)}\mathrm{vol}\big(\eta^*(-K_S)-xE_2\big)dx+\epsilon_k.
\end{equation}

Let $T_P$ be the unique hyperplane section of the surface $S$ that is singular at the point~$P$.
Then we have the following four possibilities:
\begin{itemize}
\item $T_P=L_1+L_2+L_3$, where $L_1$, $L_2$ and $L_3$ are lines such that $P=L_1\cap L_2\cap L_3$;

\item $T_P=L_1+L_2+L_3$, where $L_1$, $L_2$ and $L_3$ are lines such that $L_3\not\ni P=L_1\cap L_2$;

\item $T_P=L+C$, where $L$ is a line and $C$ is a conic such that $P\in C\cap L$.

\item $T_P$ is an irreducible cubic curve.
\end{itemize}
We plan to bound the integral in \eqref{equation:bound} depending on the type of the curve $T_P$ and on the position of the point $Q\in E_1$.
First, we deal with the cases when $Q$ is contained in the proper transform of the curve~$T_P$.
We start with

\begin{lemma}
\label{lemma:mult-special-Q-blow-up-Eckardt}
Suppose that $T_P=L_1+L_2+L_3$, where $L_1$, $L_2$ and $L_3$ are lines passing through $P$.
Let $\widetilde{L}_1$, $\widetilde{L}_2$ and $\widetilde{L}_3$ be the proper transforms on $\widetilde{S}$ of the lines $L_1$, $L_2$ and $L_3$, respectively.
Suppose that $Q\in\widetilde{L}_1\cap\widetilde{L}_2\cap\widetilde{L}_3$. Then
$$
\mathrm{mult}_Q\big(\pi^*(D)\big)\leqslant\frac{17}{9}+\epsilon_k.
$$
\end{lemma}

\begin{proof}
We may assume that $Q=\widetilde{L}_1\cap E_1$.
Denote by $\widehat{L}_1$, $\widehat{L}_2$, $\widehat{L}_3$ and $\widehat{E}_1$ the proper transforms on $\widehat{S}$ of the curves
$\widetilde{L}_1$, $\widetilde{L}_2$, $\widetilde{L}_3$ and $E_1$, respectively.
Then the intersection form of the curves $\widehat{L}_1$, $\widehat{L}_2$, $\widehat{L}_3$ and $\widehat{E}_1$ is negative definite.
Moreover, we have
$$
\eta^*(-K_S)\sim_{\mathbb{Q}} \widehat{L}_1+\widehat{L}_2+\widehat{L}_3+3\widehat{E}_1+4E_2.
$$
Thus, we conclude that $\tau(E_2)=4$.
Now, using Corollary~\ref{corollary:vol-replacement-of-line-bundle}, we compute
$$
\mathrm{vol}\big(\eta^*(-K_S)-xE_2\big)=
\begin{cases}
3-\frac{x^2}{2}, & 0\leqslant x\leqslant1,\\
\frac{20-4x-x^2}{6}, & 1\leqslant x\leqslant2,\\
\frac{(4-x)^2}{3}, & 2\leqslant x\leqslant4.\\
\end{cases}
$$
Then the required result follows from \eqref{equation:bound}.
\end{proof}

\begin{lemma}
\label{lemma:Q-in-tilde-L_1-blowup-two-line-P}
Suppose that $T_P=L_1+L_2+L_3$, where $L_1$, $L_2$ and $L_3$ are lines such that $P=L_1\cap L_2$ and $P\notin L_3$.
Let $\widetilde{L}_1$ and $\widetilde{L}_2$ be the proper transforms on $\widetilde{S}$ of the lines $L_1$ and $L_2$, respectively.
Suppose that $Q=\widetilde{L}_1\cap E_1$ or $\widetilde{L}_2\cap E_1$. Then
$$
\mathrm{mult}_Q\big(\pi^*(D)\big)\leqslant\frac{49}{27}+\epsilon_k.
$$
\end{lemma}

\begin{proof}
Denote by
$\widehat{L}_1$, $\widehat{L}_2$, $\widehat{L}_3$ and $\widehat{E}_1$ the proper transforms on $\widehat{S}$ of the curves $L_1$, $L_2$, $L_3$ and $E_1$, respectively.
Then
$$
\eta^*(-K_S)\sim_\mathbb{Q} \widehat{L}_1+\widehat{L}_2+\widehat{L}_3+2\widehat{E}_1+3E_2.
$$
Since the intersection form of the curves $\widehat{L}_1$, $\widehat{L}_2$, $\widehat{L}_3$ and $\widehat{E}_1$
is semi-negative definite, we conclude that $\tau(E_2)=3$.
Then, using Corollary~\ref{corollary:vol-replacement-of-line-bundle}, we get
$$
\mathrm{vol}\big(\eta^*(-K_S)-xE_2\big)=
\begin{cases}
3-\frac{x^2}{2}, & 0\leqslant x\leqslant1,\\
\frac{20-4x-x^2}{6}, & 1\leqslant x\leqslant2,\\
\frac{12-4x}{3}, & 2\leqslant x\leqslant3.\\
\end{cases}
$$
Then the required result follows from \eqref{equation:bound}.
\end{proof}

\begin{lemma}
\label{lemma:Q-in-tilde-L-blowup-one-line-P}
Suppose that $T_P=L+C$, where $L$ is a line, and $C$ is an irreducible conic.
Suppose that $L$ and $C$ meet transversally at $P$.
Denote by $\widetilde{L}$ and $\widetilde{C}$ the proper transforms on $\widetilde{S}$ of the curves $L$ and $C$, respectively.
Suppose that $Q=\widetilde{L}\cap E_1$. Then
$$
\mathrm{mult}_Q\big(\pi^*(D)\big)\leqslant\frac{9}{5}+\epsilon_k.
$$
\end{lemma}

\begin{proof}
Denote by $\widehat{L}$, $\widehat{C}$ and $\widehat{E}_1$ the proper transforms on  $\widehat{S}$ of the curves $L$, $C$ and $E_1$, respectively.
Then
$$
\eta^*(-K_S)\sim_\mathbb{Q} \widehat{L}+\widehat{C}+2\widehat{E}_1+3E_2.
$$
Since the intersection form of the curves $\widehat{L}$, $\widehat{C}$ and $\widehat{E}_1$
is negative definite, we conclude that $\tau(E_2)=3$.
Moreover, using Corollary~\ref{corollary:vol-replacement-of-line-bundle}, we get
$$
\mathrm{vol}\big(\eta^*(-K_S)-xE_2\big)=
\begin{cases}
3-\frac{x^2}{2}, & 0\leqslant x\leqslant 1,\\
\frac{20-4x-x^2}{6}, & 1\leqslant x\leqslant\frac{14}{5},\\
4(3-x)^2, & \frac{14}{5}\leqslant x\leqslant3.\\
\end{cases}
$$
Now the required assertion follows from \eqref{equation:bound}.
\end{proof}

\begin{lemma}
\label{lemma:Q-in-tilde-C-blowup-one-line-P}
Suppose that $T_P=L+C$, where $L$ is a line, and $C$ is an irreducible conic.
Suppose that $L$ and $C$ meet transversally at $P$.
Denote by $\widetilde{L}$ and $\widetilde{C}$ the proper transforms  on  $\widetilde{S}$ of the curves $L$ and $C$, respectively.
Suppose that $Q=\widetilde{C}\cap E_1$. Then
$$
\mathrm{mult}_Q\big(\pi^*(D)\big)\leqslant\frac{5}{3}+\epsilon_k.
$$
\end{lemma}

\begin{proof}
Denote by $\widehat{L}$, $\widehat{C}$ and $\widehat{E}_1$ the proper transforms on $\widehat{S}$ of the curves $L$, $C$ and $E_1$, respectively.
Then
$$
\eta^*(-K_S)\sim_\mathbb{Q} \widehat{L}+\widehat{C}+2\widehat{E}_1+3E_2.
$$
Since the intersection form of the curves $\widehat{L}$, $\widehat{C}$ and $\widehat{E}_1$
is negative definite, we conclude that $\tau(E_2)=3$.
Moreover, using Corollary~\ref{corollary:vol-replacement-of-line-bundle}, we get
$$
\mathrm{vol}\big(\eta^*(-K_S)-xE_2\big)=
\begin{cases}
3-\frac{x^2}{2}, & 0\leqslant x\leqslant 2,\\
(3-x)^2 & 2\leqslant x\leqslant3.\\
\end{cases}
$$
Now the required assertion follows from \eqref{equation:bound}.
\end{proof}

\begin{lemma}
\label{lemma:Q-in-L-C-tangential-blowup-one-line-P}
Suppose that $T_P=L+C$, where $L$ is a line and $C$ is an irreducible conic.
Suppose that $L$ and $C$ meet tangentially at $P$. Denote by $\widetilde{L}$ and $\widetilde{C}$ the proper transforms on $\widetilde{S}$ of the curves $L$ and $C$, respectively.
Suppose that $Q=E_1\cap\widetilde{L}\cap\widetilde{C}$. Then
$$
\mathrm{mult}_Q\big(\pi^*(D)\big)\leqslant\frac{17}{9}+\epsilon_k.
$$
\end{lemma}

\begin{proof}
Denote by $\widehat{L}$, $\widehat{C}$ and $\widehat{E}_1$ the proper transforms on  $\widehat{S}$ of the curves $\widetilde{L}$, $\widetilde{L}$ and $E_1$, respectively.
Then
$$
\eta^*(-K_S)\sim_\mathbb{Q} \widehat{L}+\widehat{C}+2\widehat{E}_1+4E_2.
$$
Since the intersection form of the curves $\widehat{L}$, $\widehat{C}$ and $\widehat{E}_1$
is negative definite, we conclude that $\tau(E_2)=4$.
Moreover, using Corollary~\ref{corollary:vol-replacement-of-line-bundle}, we get
$$
\mathrm{vol}\big(\eta^*(-K_S)-xE_2\big)=
\begin{cases}
3-\frac{x^2}{2}, & 0\leqslant x\leqslant1,\\
\frac{20-4x-x^2}{6}, & 1\leqslant x\leqslant2,\\
\frac{(4-x)^2}{3}, & 2\leqslant x\leqslant4.\\
\end{cases}
$$
Then the required result follows from \eqref{equation:bound}.
\end{proof}

\begin{lemma}
\label{lemma:Q-C-blowup-general-P}
Suppose that $T_P$ is an irreducible cubic.
Let $\widetilde{C}$ be the proper transform of the curve $C$ on the surface $\widetilde{S}$.
Suppose that $Q\in\widetilde{C}$. Then
$$
\mathrm{mult}_Q\big(\pi^*(D)\big)\leqslant\frac{5}{3}+\epsilon_k.
$$
\end{lemma}

\begin{proof}
Denote by $\widehat{C}$ and $\widehat{E}_1$ the proper transforms on $\widehat{S}$ of the curves $\widetilde{C}$ and $E_1$, respectively.
Then
$$
\eta^*(-K_S)\sim_\mathbb{Q}\widehat{C}+2\widehat{E}_1+3E_2.
$$
This gives $\tau(E_2)=3$, because the intersection form of the curves $\widehat{C}$ and $\widehat{E}_1$ is negative definite.
Using Corollary~\ref{corollary:vol-replacement-of-line-bundle}, we get
$$
\mathrm{vol}\big(\eta^*(-K_S)-xE_2\big)=
\begin{cases}
3-\frac{x^2}{2}, & 0\leqslant x\leqslant2,\\
(3-x)^2, & 2\leqslant x\leqslant3.\\
\end{cases}
$$
Then the required result follows from \eqref{equation:bound}.
\end{proof}

Now we consider the cases when $Q$ is not contained in the proper transform of the singular curve $T_P$ on the surface $\widetilde{S}$.
We start with

\begin{lemma}
\label{lemma:mult-gegeral-Q-blow-up-Eckardt}
Suppose that $T_P=L_1+L_2+L_3$, where $L_1$, $L_2$ and $L_3$ are lines passing through $P$.
Let $\widetilde{L}_1$, $\widetilde{L}_2$ and $\widetilde{L}_3$ be the proper transforms on  $\widetilde{S}$ of the lines $L_1$, $L_2$ and $L_3$, respectively.
Suppose that $Q\notin\widetilde{L}_1\cup\widetilde{L}_2\cup\widetilde{L}_3$.
Then
$$
\mathrm{mult}_Q\big(\pi^*(D)\big)\leqslant\frac{5}{3}+\epsilon_k.
$$
\end{lemma}

\begin{proof}
Denote by $\widehat{L}_1$, $\widehat{L}_2$, $\widehat{L}_3$ and $\widehat{E}_1$ the proper transforms on $\widehat{S}$ of the curves
$\widetilde{L}_1$, $\widetilde{L}_2$, $\widetilde{L}_3$ and $E_1$, respectively.
Then
$$
\eta^*(-K_S)\sim_\mathbb{Q} \widehat{L}_1+\widehat{L}_2+\widehat{L}_3+3\widehat{E}_1+3E_2.
$$
This gives $\tau(E_2)=3$, because the intersection form of the curves $\widehat{L}_1$, $\widehat{L}_2$, $\widehat{L}_3$ and $\widehat{E}_1$ is negative definite.
Using Corollary~\ref{corollary:vol-replacement-of-line-bundle}, we get
$$
\mathrm{vol}\big(\eta^*(-K_S)-xE_2\big)=
\begin{cases}
3-\frac{x^2}{2}, & 0\leqslant x\leqslant2,\\
(3-x)^2, & 2\leqslant x\leqslant3.\\
\end{cases}
$$
Then the required result follows from \eqref{equation:bound}.
\end{proof}

In the remaining cases, the pseudoeffective threshold $\tau(E_2)$ is not (always) easy to compute.
There is a (birational) reason for this.
To explain it, recall from \cite{Dolgachev} that the linear system $|-K_{\widetilde{S}}|$ is free from base points and gives a morphism $\phi\colon\widetilde{S}\to\mathbb{P}^2$.
Taking its Stein factorization, we obtain a commutative diagram
$$
\xymatrix{
\widetilde{S}\ar@{->}[drr]^{\phi}\ar@{->}[d]_{\pi}\ar@{->}[rr]^{\alpha}&&\overline{S}\ar@{->}[d]^\beta\\
S\ar@{-->}[rr]_{\rho}&&\mathbb{P}^2}
$$
where $\alpha$ is a birational morphism, $\beta$ is a double cover branched over a (possibly singular) quartic curve,
and $\rho$ is a linear projection from the point $P$.
Here, the surface $\overline{S}$ is a (possibly singular) del Pezzo surface of degree $2$.
Note that the morphism $\alpha$ is biregular if and only if the curve $T_P$ is irreducible.
Moreover, if $T_P$ is reducible, then $\alpha$-exceptional curves are proper transforms of the lines on $S$ that pass through $P$.

Let $\iota$ be the Galois involution of the double cover $\beta$.
Then its action lifts to $\widetilde{S}$.
On the other hand, this action does not always descent to a (biregular) action of the surface $S$.
Nevertheless, we can always consider $\iota$ as a birational involution of the surface $S$.
This involution is known as  Geiser involution (see \cite{Dolgachev}).
It is biregular if and only if $P$ is an Eckardt point of the surface.
In this case, the curve $E_1$ is $\iota$-invariant.
However, if $P$ is not an Eckardt point, then $\iota(E_1)$ is the proper transform of the (unique) irreducible component of the curve $T_P$
that is not a line passing through $P$.
In both cases, there exists a commutative diagram
$$
\xymatrix{
&\widetilde{S}\ar@{->}[dl]_{\pi}\ar@{->}[dr]^{\nu}&\\
S\ar@{-->}[rr]_{\psi}&& S^\prime}
$$
where $S^\prime$ is a smooth cubic surface in $\mathbb{P}^3$, which is isomorphic to the surface $S$ via the involution~$\tau$,
the morphism $\nu$ is the contraction of the curve~$\iota(E_1)$,
and $\psi$ is a birational map given by the linear subsystem in $|-2K_S|$ consisting of all curves
having multiplicity at least $3$ at the point $P$.

Let $Q^\prime=\nu(Q)$ and $P^\prime=\nu(\iota(E_1))$.
Denote by $T^\prime_Q$ the unique hyperplane section of the cubic surface $S^\prime$ that is singular at $Q^\prime$.
If $P$ is not an Eckardt point and $Q$ is not contained in the proper transform of the curve $T_P$, then $Q^\prime\ne P^\prime$.
In this case, the number $\tau(E_2)$ can be computed using $T^\prime_Q$.
This explains why the remaining cases are (slightly) more complicated.

\begin{lemma}
\label{lemma:Q-general-blowup-two-line-P}
Suppose that $T_P=L_1+L_2+L_3$, where $L_1$, $L_2$ and $L_3$ are lines such that $P=L_1\cap L_2$ and $P\notin L_3$.
Let $\widetilde{L}_1$, $\widetilde{L}_2$ and $\widetilde{L}_3$ be the proper transforms on $\widetilde{S}$ of the lines $L_1$, $L_2$ and $L_3$, respectively.
Suppose that $Q\notin \widetilde{L}_1\cup\widetilde{L}_2$. Then
$$
\mathrm{mult}_Q\big(\pi^*(D)\big)\leqslant\frac{5}{3}+\epsilon_k.
$$
\end{lemma}

\begin{proof}
Denote by
$\widehat{L}_1$, $\widehat{L}_2$, $\widehat{L}_3$ and $\widehat{E}_1$ the proper transforms on $\widehat{S}$ of the curves $L_1$, $L_2$, $L_3$ and $E_1$, respectively.
Then
$$
\eta^*(-K_S)\sim_\mathbb{Q} \widehat{L}_1+\widehat{L}_2+\widehat{L}_3+2\widehat{E}_1+2E_2,
$$
which implies that $\tau(E_2)\leqslant 2$. Using Corollary~\ref{corollary:vol-replacement-of-line-bundle-simple}, we see that
$$
\mathrm{vol}\big(\eta^*(-K_S)-xE_2\big)=3-\frac{x^2}{2}
$$
provided that $0\leqslant x\leqslant 2$.
However, we have $\tau(E_2)>2$, because the intersection form of
the curves $\widehat{L}_1$, $\widehat{L}_2$, $\widehat{L}_3$ and $\widehat{E}_1$ is not semi-negative definite.
This also follows from the fact that $\mathrm{vol}(\eta^*(-K_S)-2E_2)>0$.

Recall that $\nu\colon\widetilde{S}\to S^\prime$ is the contraction of the curve $\widetilde{L}_3$.
We let $L_1^\prime=\nu(\widetilde{L}_1)$, $L_2^\prime=\nu(\widetilde{L}_2)$ and $E_1^\prime=\nu(E_1)$.
Then $L^\prime_1$, $L^\prime_2$ and $E_1^\prime$ are coplanar lines on $S^\prime$.

Since $Q^\prime\in E_1^\prime$, the line $E_1^\prime$ is an irreducible component of the curve $T^\prime_Q$.
Thus, either $T^\prime_Q$ consists of three lines, or $T^\prime_Q$ is a union of the line $E_1^\prime$  and an irreducible conic.

Suppose that $T^\prime_Q=E_1^\prime+Z^\prime$, where $Z^\prime$ is an irreducible conic on $S^\prime$.
Then $Q^\prime\in E_1^\prime\cap Z^\prime$ and $Z^\prime\sim L^\prime_1+L^\prime_2$,
which implies that the conic $Z^\prime$ does not meet the lines $L^\prime_1$ and $L^\prime_2$.
Denote by $\widehat{Z}$ the proper transform of the conic $Z^\prime$ on the surface $\widehat{S}$.
We have
$$
\eta^*(-K_S)\sim_\mathbb{Q} \frac{1}{2}\Big(\widehat{Z}+\widehat{L}_1+\widehat{L}_2\Big)+2\widehat{E}_1+\frac{5}{2}E_2.
$$
This implies that $\tau(E_2)=\frac{5}{2}$, because the intersection form of the curves
$\widehat{Z}$, $\widehat{L}_1$, $\widehat{L}_2$ and $\widehat{E}_1$ is semi-negative definite.
Using this $\mathbb{Q}$-rational equivalence and Corollary~\ref{corollary:vol-replacement-of-line-bundle},
we compute
$$
\mathrm{vol}\big(\eta^*(-K_S)-xE_2\big)=
\begin{cases}
3-\frac{x^2}{2}, & 0\leqslant x\leqslant 2,\\
5-2x, & 2\leqslant x\leqslant\frac{5}{2}.\\
\end{cases}
$$
Thus, a direct computation and \eqref{equation:bound} give
$$
\mathrm{mult}_Q\big(\pi^*(D)\big)\leqslant\frac{59}{36}+\epsilon_k<\frac{5}{3}+\epsilon_k,
$$
which gives the required assertion.

To complete the proof, we may assume that $T^\prime_Q=E_1^\prime+M^\prime+N^\prime$, where
$M^\prime$ and $N^\prime$ are two lines on $S^\prime$ such that $Q^\prime=E^\prime_1\cap M^\prime$.
Then $M^\prime+N^\prime\sim L^\prime_1+L^\prime_2$, which implies that the lines
$M^\prime$ and $N^\prime$ do not meet the lines $L^\prime_1$ and $L^\prime_2$.
Denote by $\widehat{M}$ and $\widehat{N}$ the proper transforms on the surface $\widehat{S}$ of the lines $M^\prime$ and $N^\prime$, respectively.

Suppose that $Q^\prime$ is also contained in the line $N^\prime$.
This simply means that $Q^\prime$ is an Eckardt point of the surface $S^\prime$.
Then
$$
\eta^*(-K_S)\sim_\mathbb{Q}\frac{1}{2}\Big(\widehat{M}+\widehat{N}+\widehat{L}_1+\widehat{L}_2\Big)+2\widehat{E}_1+3E_2.
$$
This gives $\tau(E_2)\geqslant 3$.
In fact, we have  $\tau(E_2)=3$ here, because the intersection form of the curves
$\widehat{M}$, $\widehat{N}$, $\widehat{L}_1$, $\widehat{L}_2$, $\widehat{E}_1$ is negative definite.
Using Corollary~\ref{corollary:vol-replacement-of-line-bundle}, we get
$$
\mathrm{vol}\big(\eta^*(-K_S)-xE_2\big)=
\begin{cases}
3-\frac{x^2}{2}, & 0\leqslant x\leqslant 2,\\
(3-x)^2 & 2\leqslant x\leqslant3.\\
\end{cases}
$$
Now, direct computations and \eqref{equation:bound} give the required inequality.

To complete the proof the lemma, we have to consider the case $Q^\prime\notin N^\prime$.
Then
$$
\eta^*(-K_S)\sim_\mathbb{Q}\frac{1}{2}\Big(\widehat{M}+\widehat{N}+\widehat{L}_1+\widehat{L}_2\Big)+2\widehat{E}_1+\frac{5}{2}E_2.
$$
In particular, we see that $\tau(E_2)\geqslant\frac{5}{2}$.
Using this $\mathbb{Q}$-rational equivalence and Corollary~\ref{corollary:vol-replacement-of-line-bundle},
we compute
$$
\mathrm{vol}\big(\eta^*(-K_S)-xE_2\big)=
\begin{cases}
3-\frac{x^2}{2}, & 0\leqslant x\leqslant 2,\\
7-4x+\frac{x^2}{2}, & 2\leqslant x\leqslant \frac{5}{2}.\\
\end{cases}
$$
Thus, in particular, we have $\tau(E_2)>\frac{5}{2}$, since
$$
\mathrm{vol}\Big(\eta^*(-K_S)-\frac{5}{2}E_2\Big)=\frac{1}{8}.
$$

As in the previous cases, we can find $\tau(E_2)$ and compute $\mathrm{vol}(\eta^*(-K_S)-xE_2)$ for $x>\frac{5}{2}$.
However, we can avoid doing this.
Namely, note that the divisor $\widehat{E}_1+2\widehat{N}+\widehat{M}$ is nef and
$$
\Big(\widehat{E}_1+2\widehat{N}+\widehat{M}\Big)\cdot\Big(\eta^*(-K_S)-xE_2\Big)=6-2x,
$$
so that $\tau(E_2)\leqslant 3$.
Therefore, using \eqref{equation:bound} and Lemma~\ref{lemma:simple-inequality}, we see that
\begin{multline*}
\mathrm{mult}_Q\big(\pi^*(D)\big)\leqslant\frac{1}{3}\int_0^{\tau(E_2)}\mathrm{vol}\big(\eta^*(-K_S)-xE_2\big)+\epsilon_k=\\
=\frac{1}{3}\int_0^{\frac{5}{2}}\mathrm{vol}\big(\eta^*(-K_S)-xE_2\big)+\frac{1}{3}\int_{\frac{5}{2}}^{\tau(E_2)}\mathrm{vol}\big(\eta^*(-K_S)-xE_2\big)+\epsilon_k=\\
=\frac{79}{48}+\frac{1}{3}\int_{\frac{5}{2}}^{\tau(E_2)}\mathrm{vol}\big(\eta^*(-K_S)-xE_2\big)+\epsilon_k\leqslant \frac{79}{48}+\frac{\tau(E_2)-\frac{5}{2}}{3}\mathrm{vol}\Big(\eta^*(-K_S)-\frac{5}{2}E_2\Big)+\epsilon_k=\\
=\frac{79}{48}+\frac{\tau(E_2)-\frac{5}{2}}{24}+\epsilon_k\leqslant \frac{79}{48}+\frac{1}{48}+\epsilon_k=\frac{5}{3}+\epsilon_k.
\end{multline*}
This finish the proof of the lemma.
\end{proof}

\begin{lemma}
\label{lemma:Q-general-blowup-one-line-P}
Suppose that $T_P=L+C$, where $L$ is a line and $C$ is an irreducible conic.
Denote by $\widetilde{L}$ and $\widetilde{C}$ the proper transforms on $\widetilde{S}$ of the curves $L$ and $C$, respectively.
Suppose that $Q\notin\widetilde{L}\cup\widetilde{C}$. Then
$$
\mathrm{mult}_Q(\pi^*(D))\leqslant\frac{5}{3}+\epsilon_k.
$$
\end{lemma}

\begin{proof}
Denote by $\widehat{L}$, $\widehat{C}$ and $\widehat{E}_1$ the proper transforms on $\widehat{S}$ of the curves $L$, $\widetilde{C}$ and $E_1$, respectively.
Then
$$
\eta^*(-K_S)\sim_\mathbb{Q}\widehat{L}+\widehat{C}+2\widehat{E}_1+2E_2,
$$
so that $\tau(E_2)\geqslant 2$. Using Corollary~\ref{corollary:vol-replacement-of-line-bundle-simple}, we see that
$$
\mathrm{vol}\big(\eta^*(-K_S)-xE_2\big)=3-\frac{x^2}{2}
$$
provided that $0\leqslant x\leqslant 2$. Since $\mathrm{vol}(\eta^*(-K_S)-2E_2)>0$, we see that $\tau(E_2)>2$.

Recall that $\nu\colon \widetilde{S}\to  S^\prime$ is the contraction of the curve $\widetilde{C}$.
Let $L^\prime=\nu(\widetilde{L})$ and $E_1^\prime=\nu(E_1)$.
Then $L^\prime$ is a line and $E^\prime_1$ is a conic on $S^\prime$ such that $P^\prime\in L^\prime\cap E_1^\prime$.

First, we suppose that $T^\prime_Q$ is irreducible.
Denote by $\widehat{T}_Q$ the proper transform of the cubic $T^\prime_Q$ on the surface $\widehat{S}$.
Then $\widehat{T}_Q\cdot\widehat{E}_1=0$ and
$$
\widehat{T}_Q\cdot\widehat{L}=\widehat{E}_1\cdot\widehat{L}=1.
$$
Since $\widehat{L}^2=\widehat{E}_1^2=-2$ and $\widehat{T}_Q^2=-1$,
we see that the intersection form of the curves $\widehat{L}$, $\widehat{T}_Q$ and $\widehat{E}_1$
is negative definite.
On the other hand, we have
$$
\eta^*(-K_S)\sim_\mathbb{Q} \frac{1}{2}\Big(\widehat{T}_Q+\widehat{L}\Big)+\frac{3}{2}\widehat{E}_1+\frac{5}{2}E_2.
$$
This shows that $\tau(E_2)=\frac{5}{2}$.
Hence, using Corollary~\ref{corollary:vol-replacement-of-line-bundle}, we get
$$
\mathrm{vol}\big(\eta^*(-K_S)-xE_2\big)=
\begin{cases}
3-\frac{x^2}{2}, & 0\leqslant x\leqslant2,\\
\frac{44-8x-4x^2}{12}, & 2\leqslant x\leqslant\frac{17}{7},\\
4(5-2x)^2, & \frac{17}{7}\leqslant x\leqslant\frac{5}{2}.\\
\end{cases}
$$
Then a direct calculation and  \eqref{equation:bound} give
$$
\mathrm{mult}_Q\big(\pi^*(D)\big)\leqslant\frac{103}{63}+\epsilon_k<\frac{5}{3}+\epsilon_k.
$$

Now we suppose that $T^\prime_Q=\ell^\prime+Z^\prime$, where $\ell^\prime$ is a line, and $Z^\prime$ is an irreducible conic.
Denote by $\widehat{\ell}$ and $\widehat{Z}$ the proper transforms on $\widehat{S}$ of the curves $\ell^\prime$ and $Z^\prime$, respectively.
We get
$$
\eta^*(-K_S)\sim_\mathbb{Q}\frac{1}{2}\Big(\widehat{\ell}+\widehat{Z}+\widehat{L}\Big)+\frac{3}{2}\widehat{E}_1+\frac{5}{2}E_2.
$$
which implies that $\tau(E_2)\geqslant \frac{5}{2}$.
Using Corollary~\ref{corollary:vol-replacement-of-line-bundle}, we get
$$
\mathrm{vol}\big(\eta^*(-K_S)-xE_2\big)=
\begin{cases}
3-\frac{x^2}{2}, & 0\leqslant x\leqslant2,\\
\frac{34-16x+x^2}{6}, & 2\leqslant x\leqslant\frac{5}{2}.
\end{cases}
$$
In particular, we have
$$
\mathrm{vol}\Big(\eta^*(-K_S)-\frac{5}{2}E_2\Big)=\frac{1}{24},
$$
which implies that $\tau(E_2)>\frac{5}{2}$.
Observe that the divisor $\widehat{\ell}+2\widehat{Z}+\widehat{L}$ is nef and
$$
\big(\widehat{\ell}+2\widehat{Z}+\widehat{L}\big)\cdot\big(\eta^*(-K_S)-xE_2\big)=9-3x,
$$
which implies that $\tau(E_2)\leqslant 3$.
Thus, using \eqref{equation:bound} and Lemma~\ref{lemma:simple-inequality}, we get
\begin{multline*}
\mathrm{mult}_Q\big(\pi^*(D)\big)\leqslant\frac{1}{3}\int_0^{\tau(E_2)}\mathrm{vol}\big(\eta^*(-K_S)-xE_2\big)+\epsilon_k=\\
=\frac{1}{3}\int_0^{\frac{5}{2}}\mathrm{vol}\big(\eta^*(-K_S)-xE_2\big)+\frac{1}{3}\int_{\frac{5}{2}}^{\tau(E_2)}\mathrm{vol}\big(\eta^*(-K_S)-xE_2\big)+\epsilon_k=\\
=\frac{709}{432}+\frac{1}{3}\int_{\frac{5}{2}}^{\tau(E_2)}\mathrm{vol}\big(\eta^*(-K_S)-xE_2\big)+\epsilon_k\leqslant\frac{709}{432}+\frac{\tau(E_2)-\frac{5}{2}}{3}\mathrm{vol}\Big(\eta^*(-K_S)-\frac{5}{2}E_2\Big)+\epsilon_k=\\
=\frac{709}{432}+\frac{\tau(E_2)-\frac{5}{2}}{48}+\epsilon_k\leqslant\frac{709}{432}+\frac{1}{96}+\epsilon_k=\frac{89}{54}+\epsilon_k<\frac{5}{3}+\epsilon_k.
\end{multline*}

To complete the proof of the lemma, we may assume that $T^\prime_Q=\ell^\prime+M^\prime+N^\prime$,
where $\ell^\prime$, $M^\prime$ and $N^\prime$ are lines such that $Q^\prime\in M^\prime\cap N^\prime$.
Since $E^\prime_1$ is a conic passing through $Q^\prime$,
we conclude that $Q^\prime$ is not contained in the line $\ell^\prime$.
Note that $\ell^\prime\neq L^\prime$, and the lines $\ell^\prime$, $M^\prime$ and $N^\prime$ do not pass through $P^\prime$.

Denote by $\widehat{\ell}$, $\widehat{M}$ and $\widehat{N}$ the proper transforms on $\widehat{S}$ of the lines $\ell^\prime$, $M^\prime$ and $N^\prime$, respectively.
We get
$$
\eta^*(-K_S)\sim_{\mathbb{Q}}\frac{1}{2}\Big(\widehat{\ell}+\widehat{M}+\widehat{N}+\widehat{L}\Big)+\frac{3}{2}\widehat{E}_1+\frac{5}{2}E_2,
$$
which implies that $\tau(E_2)\geqslant \frac{5}{2}$.
In fact, we have $\tau(E_2)>\frac{5}{2}$,
because the intersection form of the curves $\widehat{\ell}$, $\widehat{M}$, $\widehat{N}$, $\widehat{L}$ and $\widehat{E}_1$
is not semi-negative definite.
Nevertheless, we can use Corollary~\ref{corollary:vol-replacement-of-line-bundle} to compute
$$
\mathrm{vol}\big(\eta^*(-K_S)-xE_2\big)=
\begin{cases}
3-\frac{x^2}{2}, & 0\leqslant x\leqslant2,\\
\frac{92-56x+8x^2}{12}, & 2\leqslant x\leqslant\frac{5}{2},\\
\end{cases}
$$
so that, in particular, we have
$$
\mathrm{vol}\Big(\eta^*(-K_S)-\frac{5}{2}E_2\Big)=\frac{1}{6}.
$$
Observe that the divisor $2\widehat{\ell}+\widehat{M}+\widehat{N}$ is nef and
$$
\big(2\widehat{\ell}+\widehat{M}+\widehat{N}\big)\cdot\big(\eta^*(-K_S)-xE_2\big)=6-2x,
$$
which implies that $\tau(E_2)\leqslant 3$.
Thus, using \eqref{equation:bound} and Lemma~\ref{lemma:barycenter-inequality}, we get
\begin{multline*}
\mathrm{mult}_Q\big(\pi^*(D)\big)\leqslant\frac{1}{3}\int_0^{\tau(E_2)}\mathrm{vol}\big(\eta^*(-K_S)-xE_2\big)+\epsilon_k=\\
=\frac{1}{3}\int_0^{\frac{5}{2}}\mathrm{vol}\big(\eta^*(-K_S)-xE_2\big)+\frac{1}{3}\int_{\frac{5}{2}}^{\tau(E_2)}\mathrm{vol}\big(\eta^*(-K_S)-xE_2\big)+\epsilon_k=\\
=\frac{89}{54}+\frac{1}{3}\int_{\frac{5}{2}}^{\tau(E_2)}\mathrm{vol}\big(\eta^*(-K_S)-xE_2\big)+\epsilon_k\leqslant
\frac{89}{54}+\frac{2}{9}\Big(\tau(E_2)-\frac{5}{2}\Big)\mathrm{vol}\Big(\eta^*(-K_S)-\frac{5}{2}E_2\Big)+\epsilon_k=\\
=\frac{89}{54}+\frac{2}{54}\Big(\tau(E_2)-\frac{5}{2}\Big)+\epsilon_k\leqslant\frac{89}{54}+\frac{1}{54}+\epsilon_k=\frac{5}{3}+\epsilon_k.
\end{multline*}
The proof is complete.
\end{proof}

\begin{lemma}
\label{lemma:Q-general-nodal-blowup-general-P}
Suppose that $T_P$ is an irreducible cubic curve.
Let $\widetilde{C}$ be its proper transform on the surface $\widetilde{S}$.
Suppose that $Q\notin\widetilde{C}$. Then
$$
\mathrm{mult}_Q\big(\pi^*(D)\big)\leqslant\frac{5}{3}+\epsilon_k.
$$
\end{lemma}

\begin{proof}
Denote by $\widehat{C}$ and $\widehat{E}_1$ the proper transforms on $\widehat{S}$ of the curves $\widetilde{C}$ and $E_1$, respectively.
Then
$$
\eta^*(-K_S)\sim_\mathbb{Q}\widehat{C}+2\widehat{E}_1+2E_2.
$$
Thus, using Corollary~\ref{corollary:vol-replacement-of-line-bundle-simple}, we get
$\mathrm{vol}(\eta^*(-K_S)-xE_2)=3-\frac{x^2}{2}$ provided that $0\leqslant x\leqslant 2$.

Recall that $\nu\colon \widetilde{S}\to  S^\prime$ is the contraction of the curve $\widetilde{C}$.
Let $E^\prime=\nu(E_1)$. Then $E_1^\prime$ is an irreducible cubic curve that is singular at $P^\prime$.
Thus, the curve $E_1^\prime$ is smooth at the point $Q^\prime$, so that $T^\prime_Q\neq E_1^\prime$.
One can easily check that $T^\prime_Q$ does not contain $P^\prime$.

Suppose that $T^\prime_Q$ is an irreducible cubic.
Denote by $\widehat{T}_Q$ the proper transform of the curve $T^\prime_Q$ on the surface $\widehat{S}$.
We get $\widehat{E}_1^2=-2$, $\widehat{T}_Q^2=-1$, $\widehat{E}_1\cdot \widehat{T}_Q=1$ and
$$
\eta^*(-K_S)\sim_\mathbb{Q} \frac{1}{2}\widehat{T}_Q+\frac{3}{2}\widehat{E}_1+\frac{5}{2}E_2,
$$
which implies that $\tau(E_2)=\frac{5}{2}$.
Using Corollary~\ref{corollary:vol-replacement-of-line-bundle}, we get
$$
\mathrm{vol}\big(\eta^*(-K_S)-xE_2\big)=
\begin{cases}
3-\frac{x^2}{2}, & 0\leqslant x\leqslant\frac{12}{5},\\
3(5-2x)^2, & \frac{12}{5}\leqslant x\leqslant \frac{5}{2}.\\
\end{cases}
$$
Then \eqref{equation:bound} and direct calculations give
$$
\mathrm{mult}_Q\big(\pi^*(D)\big)\leqslant\frac{49}{30}+\epsilon_k<\frac{5}{3}+\epsilon_k.
$$

Now we suppose that
$T^\prime_Q=\ell^\prime+Z^\prime$, where $\ell^\prime$ is a line and $Z^\prime$ is an irreducible conic.
Denote by $\widehat{\ell}$ and $\widehat{Z}$ the proper transforms on $\widehat{S}$ of the curves $\ell^\prime_Q$ and $Z^\prime$, respectively.
We get
$$
\eta^*(-K_S)\sim_\mathbb{Q}\frac{1}{2}\Big(\widehat{\ell}+\widehat{Z}\Big)+\frac{3}{2}\widehat{E}_1+\frac{5}{2}E_2.
$$
Since the intersection form of the curves $\widehat{\ell}$, $\widehat{Z}$ and $\widehat{E}_1$
is semi-negative definite, we conclude that $\tau(E_2)=\frac{5}{2}$.
Using Corollary~\ref{corollary:vol-replacement-of-line-bundle}, we get
$$
\mathrm{vol}\big(\eta^*(-K_S)-xE_2\big)=
\begin{cases}
3-\frac{x^2}{2}, & 0\leqslant x\leqslant2,\\
5-2x, & 2\leqslant x\leqslant\frac{5}{2}.\\
\end{cases}
$$
Hence, using  \eqref{equation:bound}, we see that
$$
\mathrm{mult}_Q\big(\pi^*(D)\big)\leqslant\frac{59}{36}+\epsilon_k<\frac{5}{3}+\epsilon_k.
$$

To complete the proof, we may assume that
$T^\prime_Q=\ell^\prime+M^\prime+N^\prime$, where $\ell^\prime$, $M^\prime$ and $N^\prime$ are lines
such that $Q^\prime\in M^\prime\cap N^\prime$.
Denote by $\widehat{\ell}$, $\widehat{M}$ and $\widehat{N}$ the proper transforms on $\widehat{S}$ of the lines $\ell^\prime$, $M^\prime$ and $N^\prime$, respectively.
If $Q^\prime$ is contained in the line $\ell^\prime$, then
$$
\eta^*(-K_S)\sim_\mathbb{Q}\frac{1}{2}\Big(\widehat{\ell}+\widehat{M}+\widehat{N}\Big)+\frac{3}{2}\widehat{E}_1+3E_2,
$$
and  the intersection form of the curves $\widehat{\ell}$, $\widehat{M}$, $\widehat{N}$ and $\widehat{E}_1$
is negative definite, which implies that $\tau(E_2)=3$.
In this case, Corollary~\ref{corollary:vol-replacement-of-line-bundle} gives
$$
\mathrm{vol}\big(\eta^*(-K_S)-xE_2\big)=
\begin{cases}
3-\frac{x^2}{2}, & 0\leqslant x\leqslant2,\\
(3-x)^2, & 2\leqslant x\leqslant 3,\\
\end{cases}
$$
which implies the required inequality by \eqref{equation:bound}.

To complete the proof, we may assume that  $Q^\prime$ is not contained in $\ell^\prime$.
Then the intersection form of the curves $\widehat{\ell}$, $\widehat{M}$, $\widehat{N}$ and $\widehat{E}_1$ is not semi-negative definite.
Since
$$
\eta^*(-K_S)\sim_\mathbb{Q}\frac{1}{2}\Big(\widehat{\ell}+\widehat{M}+\widehat{N}\Big)+\frac{3}{2}\widehat{E}_1+\frac{5}{2}E_2,
$$
we conclude that $\tau(E_2)>\frac{5}{2}$.
Moreover, using Corollary~\ref{corollary:vol-replacement-of-line-bundle}, we get
$$
\mathrm{vol}\big(\eta^*(-K_S)-xE_2\big)=
\begin{cases}
3-\frac{x^2}{2}, & 0\leqslant x\leqslant2,\\
\frac{x^2-8x+14}{2}, & 2\leqslant x\leqslant\frac{5}{2}.\\
\end{cases}
$$
In particular, we have
$$
\mathrm{vol}\Big(\eta^*(-K_S)-\frac{5}{2}E_2\Big)=\frac{1}{8}.
$$
Observe that the divisor $2\widehat{\ell}+\widehat{M}+\widehat{N}$ is nef and
$$
\big(2\widehat{\ell}+\widehat{M}+\widehat{N}\big)\cdot\big(\eta^*(-K_S)-xE_2\big)=6-2x,
$$
which implies that $\tau(E_2)\leqslant 3$.
Thus, using \eqref{equation:bound} and Lemma~\ref{lemma:simple-inequality}, we get
\begin{multline*}
\mathrm{mult}_Q\big(\pi^*(D)\big)\leqslant\frac{1}{3}\int_0^{\tau(E_2)}\mathrm{vol}\big(\eta^*(-K_S)-xE_2\big)+\epsilon_k=\\
=\frac{1}{3}\int_0^{\frac{5}{2}}\mathrm{vol}\big(\eta^*(-K_S)-xE_2\big)+\frac{1}{3}\int_{\frac{5}{2}}^{\tau(E_2)}\mathrm{vol}\big(\eta^*(-K_S)-xE_2\big)+\epsilon_k=\\
=\frac{79}{48}+\frac{1}{3}\int_{\frac{5}{2}}^{\tau(E_2)}\mathrm{vol}\big(\eta^*(-K_S)-xE_2\big)+\epsilon_k\leqslant
\frac{79}{48}+\frac{\tau(E_2)-\frac{5}{2}}{3}\mathrm{vol}\Big(\eta^*(-K_S)-\frac{5}{2}E_2\Big)+\epsilon_k=\\
=\frac{79}{48}+\frac{\tau(E_2)-\frac{5}{2}}{24}+\epsilon_k\leqslant\frac{79}{48}+\frac{1}{48}+\epsilon_k=\frac{5}{3}+\epsilon_k.
\end{multline*}
This completes the proof of the lemma.
\end{proof}

Using Corollary~\ref{corollary:mult-at-Q-lc-at-P} and Lemmas~\ref{lemma:mult-special-Q-blow-up-Eckardt},
\ref{lemma:Q-in-tilde-L_1-blowup-two-line-P}, \ref{lemma:Q-in-tilde-L-blowup-one-line-P}, \ref{lemma:Q-in-tilde-C-blowup-one-line-P}, \ref{lemma:Q-in-L-C-tangential-blowup-one-line-P},
\ref{lemma:Q-C-blowup-general-P}, \ref{lemma:mult-gegeral-Q-blow-up-Eckardt},
\ref{lemma:Q-general-blowup-two-line-P}, \ref{lemma:Q-general-blowup-one-line-P}, \ref{lemma:Q-general-nodal-blowup-general-P},
we immediately get

\begin{corollary}
\label{corollary:delta-18-17}
We have $\delta(S)\geqslant\frac{18}{17}$.
\end{corollary}

\section{Proof of the main result}
\label{section:delta-bigger-than-one}

In this section, we prove Theorem~\ref{theorem:main}.
Let $S$ be a smooth cubic surface. We have to prove that $\delta(S)\geqslant\frac{6}{5}$.
Fix a positive rational number $\lambda<\frac{6}{5}$.
Let $D$ be a $k$-basis type divisor.
To prove Theorem \ref{theorem:main}, it is enough to show that, the log pair $(S,\lambda D)$ is log canonical for $k\gg 1$.
Suppose that this is not the case.
Then there exists a point $P\in S$ such that $(S,\lambda D)$ is not log canonical at $P$ for $k\gg 1$.
Let us seek for a contradiction using results obtained in Section~\ref{section:mult-estimates}.

Let $\pi\colon\widetilde{S}\to S$ be the blowup of the point $P$, and let $E_1$ be the exceptional divisor of the blow up $\pi$.
Denote by $\widetilde{D}$ the proper transform of $D$ via $\pi$. Then
$$
K_{\widetilde{S}}+\lambda\widetilde{D}+\big(\lambda\mathrm{mult}_P(D)-1\big)E_1\sim_{\mathbb{Q}}\pi^*\big(K_S+\lambda D\big).
$$
By Corollary~\ref{corollary:lc-invariant-under-blowup}, the log pair $(\widetilde{S},\lambda\widetilde{D}+(\lambda\mathrm{mult}_P(D)-1)E_1)$
is not log canonical at some point $Q\in E_1$.
Thus, using Lemma \ref{lemma:not-lc-mult-large}, we see that
\begin{equation}
\label{equation:mult-5-3}
\mathrm{mult}_Q\big(\pi^*\big(D\big)\big)=\mathrm{mult}_P\big(D\big)+\mathrm{mult}_Q\big(\widetilde{D}\big)>\frac{2}{\lambda}>\frac{5}{3}.
\end{equation}

Let $\sigma\colon\widehat{S}\to\widetilde{S}$ be the blowup of the point $Q$, and let $E_2$ be the exceptional curve of $\sigma$.
Denote by $\widehat{D}$ and $\widehat{E}_1$ the proper transforms on $\widehat{S}$ of the divisors $\widetilde{D}$ and $E_1$, respectively.
By Corollary~\ref{corollary:lc-invariant-under-blowup}, the log pair
$$
\Big(\widehat{S},\lambda\widehat{D}+\big(\lambda\mathrm{mult}_P(D)-1\big)\widehat{E}_1+\big(\lambda\mathrm{mult}_P(D)+\lambda\mathrm{mult}_Q(\widetilde{D})-2\big)E_2\Big)
$$
is not log canonical at some point $O\in E_2$.

Let $T_P$ be the hyperplane section of the surface $S$ that is singular at $P$.
Then $T_P$ must be reducible.
This follows from \eqref{equation:mult-5-3} and Lemmas~\ref{lemma:Q-C-blowup-general-P} and \ref{lemma:Q-general-nodal-blowup-general-P}.

Denote by $\widetilde{T}_P$ the proper transform of the curve $T_P$ on the surface $\widetilde{S}$.
Then $Q\in\widetilde{T}_P$.
This follows from \eqref{equation:mult-5-3} and Lemmas~\ref{lemma:Q-general-blowup-two-line-P} and \ref{lemma:Q-general-blowup-one-line-P}.

In the remaining part of this section, we will deal with the following four cases:
\begin{enumerate}
\item $T_P$ is a union of three lines passing through $P$;
\item $T_P$ is a union of three lines and only two of them pass through $P$;
\item $T_P$ is a union of line and a conic that intersect transversally at $P$;
\item $T_P$ is a union of line and a conic that intersect tangentially at $P$.
\end{enumerate}
We will treat each of them in a separate subsection. We start with

\subsection{Case 1}
\label{subsection:1}
We have $T_P=L_1+L_2+L_3$, where $L_1$, $L_2$ and $L_3$ are lines passing through the point $P$.
We write
$$
\lambda D=a_1L_1+a_2L_2+a_3L_3+\Omega,
$$
where $a_1$, $a_2$ and $a_3$ are nonnegative rational numbers,
and $\Omega$ is an effective $\mathbb{Q}$-divisor whose support does not contain $L_1$, $L_2$ or $L_3$.
Then
\begin{equation}
\label{equation:three-line-relation-Eckardt}
L_1\cdot\Omega=\lambda+a_1-a_2-a_3.
\end{equation}

Denote by $\widetilde{L}_1$, $\widetilde{L}_2$ and $\widetilde{L}_3$ the proper transforms on $\widetilde{S}$ of the lines $L_1$, $L_2$ and $L_3$, respectively.
We know that $Q\in\widetilde{L}_1\cup\widetilde{L}_2\cup\widetilde{L}_3$, so that we may assume that $Q=\widetilde{L}_1\cap E_1$.
Let $\widetilde{\Omega}$ be the proper transform of the divisor $\Omega$ on the surface $\widetilde{S}$, and let $m=\mathrm{mult}_P(\Omega)$.
Then the log pair
$$
\Big(\widetilde{S},a_1\widetilde{L}_1+\widetilde{\Omega}+\big(a_1+a_2+a_3+m-1\big)E_1\Big)
$$
is not log canonical at the point $Q$.

By Lemma~\ref{lemma:mult-of-a-line}, we have
\begin{equation}
\label{equation:a<5/9-Eckardt}
a_1\leqslant\Big(\frac{5}{9}+\varepsilon_k\Big)\lambda<1,
\end{equation}
where $\varepsilon_k$ is a small constant depending on $k$ such that $\varepsilon_k\to 0$ as $k\to \infty$.
Thus, applying Corollary~\ref{corollary:inversion-of-adjunction}, we see that
$$
L_1\cdot\Omega+a_1+a_2+a_3-1=\widetilde{L}_1\cdot\Big(\widetilde{\Omega}+\big(a_1+a_2+a_3+m-1\big)E_1\Big)>1,
$$
which gives $L_1\cdot\Omega>2-a_1-a_2-a_3$.
Combining this with \eqref{equation:three-line-relation-Eckardt}, we get
\begin{equation}
\label{equation:a_1>-Eckardt}
a_1>\frac{2-\lambda}{2}.
\end{equation}

Let $\widetilde{m}=\mathrm{mult}_Q(\widetilde{\Omega})$. Then by Lemma~\ref{lemma:mult-special-Q-blow-up-Eckardt}, we have
\begin{equation}
\label{equation:mult-Q-special-Eckardt}
2a_1+a_2+a_3+m+\widetilde{m}\leqslant\Big(\frac{17}{9}+\epsilon_k\Big)\lambda,
\end{equation}
where $\epsilon_k$ is a small constant depending on $k$ such that $\epsilon_k\to 0$ as $k\to \infty$.
Then using \eqref{equation:a_1>-Eckardt} and $m\geqslant \widetilde{m}$, we deduce that
\begin{equation}
\label{equation:tilde-m<-Eckardt}
\widetilde{m}<\Big(\frac{13}{9}+\frac{\epsilon_k}{2}\Big)\lambda-1<1.
\end{equation}

Denote by $\widehat{L}_1$ and $\widehat{\Omega}$ the proper transforms on $\widehat{S}$ of the divisors $\widetilde{L}_1$ and $\widetilde{\Omega}$, respectively.
Then the log pair
$$
\Big(\widehat{S},a_1\widehat{L}_1+\widehat{\Omega}+\big(a_1+a_2+a_3+m-1\big)\widehat{E_1}+\big(2a_1+a_2+a_3+m+\widetilde{m}-2\big)E_2\Big)
$$
is not log canonical at the point $O$.

We claim that $O\in\widehat{L}_1\cup\widehat{E}_1$.
Indeed, we have $(2a_1+a_2+a_3+m+\widetilde{m}-2)<1$ by \eqref{equation:mult-Q-special-Eckardt}.
Thus, if $O\not\in\widehat{L}_1\cup\widehat{E}_1$, then Corollary~\ref{corollary:inversion-of-adjunction} gives
$$
\widetilde{m}=\widehat{\Omega}\cdot E_2\geqslant\big(\widehat{\Omega}\cdot E_2\big)_O>1,
$$
which is impossible by \eqref{equation:tilde-m<-Eckardt}.
Thus, we have $O\in\widehat{L}_1\cup\widehat{E}_1$.

If $O\in\widehat{E}_1$, then the log pair
$$
\Big(\widehat{S},\widehat{\Omega}+\big(a_1+a_2+a_3+m-1\big)\widehat{E_1}+\big(2a_1+a_2+a_3+m+\widetilde{m}-2\big)E_2\Big)
$$
is not log canonical at the point $O$. Then Corollary~\ref{corollary:inversion-of-adjunction} gives
$a_1+a_2+a_3+m+\widetilde{m}>2$, so that \eqref{equation:a_1>-Eckardt} and \eqref{equation:mult-Q-special-Eckardt} gives
$$
\Big(\frac{17}{9}+\epsilon_k\Big)\lambda\geqslant 2a_1+a_2+a_3+m+\widetilde{m}>2+a_1>3-\frac{\lambda}{2},
$$
which is impossible, since $\lambda<\frac{6}{5}$ and $\epsilon_k\to 0$ as $k\to \infty$.

Thus, we see that $O\in\widehat{L}_1$. Then the log pair
$$
\Big(\widehat{S},a_1\widehat{L}_1+\widehat{\Omega}+\big(2a_1+a_2+a_3+m+\widetilde{m}-2\big)E_2\Big)
$$
is not log canonical at the point $O$.
Now, using \eqref{equation:mult-Q-special-Eckardt} and \eqref{equation:tilde-m<-Eckardt}, we have
$$
\mathrm{mult}_O\Big(\widehat{\Omega}+\big(2a_1+a_2+a_3+m+\widetilde{m}-2\big)E_2\Big)=2a_1+a_2+a_3+m+2\widetilde{m}-2<\Big(\frac{10}{3}+\frac{3\epsilon_k}{2}\Big)\lambda-3<1,
$$
since $\lambda<\frac{6}{5}$ and $k\gg 1$. Thus, Lemma \ref{lemma:Kewei-inequality} gives
$$
L_1\cdot\Omega+2a_1+a_2+a_3-2=\widehat{L}_1\cdot\Big(\widehat{\Omega}+\big(2a_1+a_2+a_3+m+\widetilde{m}-2\big)E_2\Big)>2-a_1,
$$
so that $L_1\cdot\Omega+3a_1+a_2+a_3>4$.
Using \eqref{equation:three-line-relation-Eckardt} we get $\lambda+4a_1>4$.
Using \eqref{equation:a<5/9-Eckardt}, we get
$$
\Big(\frac{29}{9}-\varepsilon_k\Big)\lambda>4,
$$
which is impossible, since $\lambda<\frac{6}{5}$ and $\varepsilon_k\to 0$ as $k\to \infty$.

\subsection{Case 2}
\label{subsection:2}
We have $T_P=L_1+L_2+L_3$, where $L_1$, $L_2$ and $L_3$ are coplanar lines such that $P=L_1\cap L_2$ and $P\notin L_3$.
As in the previous case, we write
$$
\lambda D=a_1L_1+a_2L_2+\Omega,
$$
where $a_1$ and $a_2$ are nonnegative rational numbers, and $\Omega$ is an effective $\mathbb{Q}$-divisor
whose support does not contain the lines $L_1$ and $L_2$.
Then
\begin{equation}
\label{equation:two-line-relation}
L_1\cdot\Omega=\lambda+a_1-a_2.
\end{equation}

Denote by $\widetilde{L}_1$ and $\widetilde{L}_2$ the proper transforms on $\widetilde{S}$ of the lines $L_1$ and $L_2$, respectively.
We know that $Q\in\widetilde{L}_1\cup\widetilde{L}_2$, so that we may assume that $Q=\widetilde{L}_1\cap E_1$.
Let $\widetilde{\Omega}$ be the proper transform of the divisor $\Omega$ on the surface $\widetilde{S}$, and let $m=\mathrm{mult}_P(\Omega)$.
Then the log pair
$$
\Big(\widetilde{S},a_1\widetilde{L}_1+\widetilde{\Omega}+\big(a_1+a_2+a_3+m-1\big)E_1\Big)
$$
is not log canonical at the point $Q$.

By Lemma \ref{lemma:mult-of-a-line}, we have
\begin{equation}
\label{equation:a<5/9-two-line-P}
a_1\leqslant\Big(\frac{5}{9}+\varepsilon_k\Big)\lambda<1,
\end{equation}
where $\varepsilon_k$ is a small constant depending on $k$ such that $\varepsilon_k\to 0$ as $k\to \infty$.
Thus, using Corollary~\ref{corollary:inversion-of-adjunction}, we obtain $L_1\cdot\Omega>2-a_1-a_2$.
Then, using \eqref{equation:two-line-relation}, we deduce
\begin{equation}
\label{equation:a_1>-two-line-P}
a_1>\frac{2-\lambda}{2}.
\end{equation}

Let $\widetilde{m}=\mathrm{mult}_Q(\widetilde{\Omega})$.
By Lemma \ref{lemma:Q-in-tilde-L_1-blowup-two-line-P}, we have
\begin{equation}
\label{equation:mult-Q-special-two-line}
2a_1+a_2+m+\widetilde{m}\leqslant\Big(\frac{49}{27}+\epsilon_k\Big)\lambda.
\end{equation}
where $\epsilon_k$ is a small constant depending on $k$ such that $\epsilon_k\to 0$ as $k\to \infty$.
Thus, using \eqref{equation:a_1>-two-line-P} and $\widetilde{m}\leqslant m$, we deduce
\begin{equation}
\label{equation:tilde-m-two-line-P}
\widetilde{m}<\Big(\frac{38}{27}+\frac{\epsilon_k}{2}\Big)\lambda-1<1.
\end{equation}

Denote by $\widehat{L}_1$ and $\widehat{\Omega}$ the proper transforms on $\widehat{S}$ of the divisors $\widetilde{L}_1$ and $\widetilde{\Omega}$, respectively.
Then the log pair
$$
\Big(\widehat{S},a_1\widehat{L}_1+\widehat{\Omega}+\big(a_1+a_2+m-1\big)\widehat{E}_1+\big(2a_1+a_2+m+\widetilde{m}-2\big)E_2\Big)
$$
is not log canonical at the point $O$.
Then $2a_1+a_2+m+\widetilde{m}-2<1$ by \eqref{equation:mult-Q-special-two-line}.
Thus, using \eqref{equation:tilde-m-two-line-P} and arguing as in Subsection~\ref{subsection:1}, we see that $O\in\widehat{L}_1\cup\widehat{E}_1$.

If $O\in\widehat{E}_1$, then the log pair
$$
\Big(\widehat{S},\widehat{\Omega}+\big(a_1+a_2+m-1\big)\widehat{E}_1+\big(2a_1+a_2+m+\widetilde{m}-2\big)E_2\Big)
$$
is not log canonical at the point $O$, so that $a_1+a_2+m+\widetilde{m}>2$ by Corollary~\ref{corollary:inversion-of-adjunction}.
Hence, using \eqref{equation:a_1>-two-line-P} and \eqref{equation:mult-Q-special-two-line}, we get
$$
\Big(\frac{49}{27}+\epsilon_k\Big)\lambda\geqslant 2a_1+a_2+m+\widetilde{m}>2+a_1>3-\frac{\lambda}{2},
$$
which is impossible, since $\lambda<\frac{6}{5}$ and $\epsilon_k\to 0$ as $k\to \infty$.

We see that $O\in\widehat{L}_1$. Then the log pair
$$
\Big(\widehat{S},a_1\widehat{L}_1+\widehat{\Omega}+\big(2a_1+a_2+m+\widetilde{m}-2\big)E_2\Big)
$$
is not log canonical at the point $O$.
Now, using \eqref{equation:mult-Q-special-two-line} and \eqref{equation:tilde-m-two-line-P}, we deduce
$$
\mathrm{mult}_O\Big(\widehat{\Omega}+\big(2a_1+a_2+m+\widetilde{m}-2\big)E_2\Big)=
2a_1+a_2+m+2\widetilde{m}-2<\Big(\frac{29}{9}+\frac{3\epsilon_k}{2}\Big)\lambda-3<1,
$$
because $\lambda<\frac{6}{5}$ and $k\gg 1$.
Then we may apply Lemma~\ref{lemma:Kewei-inequality} to get
$$
L_1\cdot\Omega+2a_1+a_2-2=\widehat{L}_1\cdot\Big(\widehat{\Omega}+\big(2a_1+a_2+m+\widetilde{m}-2\big)E_2\Big)>2-a_1,
$$
so that $L_1\cdot\Omega+3a_1+a_2>4$.
Using \eqref{equation:two-line-relation} we get $\lambda+4a_1>4$.
Then, by \eqref{equation:a<5/9-two-line-P}, we have
$$
\Big(\frac{29}{9}-\varepsilon_k\Big)\lambda>4,
$$
which is impossible, since $\lambda<\frac{6}{5}$ and $\varepsilon_k\to 0$ as $k\to \infty$.

\subsection{Case 3}
\label{subsection:3}
We have $T_P=L+C$, where $L$ is a line and $C$ is an irreducible conic such that they intersect transversally at $P$.
As in the previous cases, we write
$$
\lambda D=aL+bC+\Omega,
$$
where $a$ and $b$ are nonnegative rational numbers, and $\Omega$ is an effective $\mathbb{Q}$-divisor whose support does not contain the curves $L$ and $C$.
Then Lemma \ref{lemma:mult-of-a-line} gives us
\begin{equation}
\label{equation:a<5/9-transversal-one-line-P}
a\leqslant\Big(\frac{5}{9}+\varepsilon_k\Big)\lambda<1,
\end{equation}
where $\varepsilon_k$ is a small constant depending on $k$ such that $\varepsilon_k\to 0$ as $k\to \infty$.
And also, we have
\begin{equation}
\label{equation:line-conic-transversal-relation}
L\cdot\Omega=\lambda+a-2b.
\end{equation}

Denote by $\widetilde{L}$ and $\widetilde{C}$ the proper transforms on $\widetilde{S}$ of the curves $L$ and $C$, respectively.
We know that $Q\in\widetilde{L}\cup\widetilde{C}$.
Moreover, using \eqref{equation:mult-5-3} and Lemma~\ref{lemma:Q-in-tilde-C-blowup-one-line-P}, we see that $Q=\widetilde{L}\cap E_1$.

Denote by $\widetilde{\Omega}$ the proper transforms on $\widetilde{S}$ of the divisor $\Omega$.
Let $m=\mathrm{mult}_P(\Omega)$. Then the log pair
$$
\Big(\widetilde{S},a\widetilde{L}+\widetilde{\Omega}+\big(a+b+m-1\big)E_1\Big)
$$
is not log canonical at $Q$.
Since $a<1$, we can apply Corollary~\ref{corollary:inversion-of-adjunction} to this log pair and the curve $\widetilde{L}$.
This gives $L\cdot\Omega>2-a-b$.
Combining this with \eqref{equation:line-conic-transversal-relation}, we have $\lambda+2a-b>2$,
so that
\begin{equation}
\label{equation:a>-transversal-one-line-P}
a>\frac{2+b-\lambda}{2}\geqslant\frac{2-\lambda}{2}.
\end{equation}

Let $\widetilde{m}=\mathrm{mult}_Q(\widetilde{\Omega})$.
Then Lemma \ref{lemma:Q-in-tilde-L-blowup-one-line-P} gives
\begin{equation}
\label{equation:mult-Q-transversal-one-line-P}
2a+b+m+\widetilde{m}\leqslant\Big(\frac{9}{5}+\epsilon_k\Big)\lambda,
\end{equation}
where $\epsilon_k$ is a small constant depending on $k$ such that $\epsilon_k\to 0$ as $k\to \infty$.
Thus, using \eqref{equation:a>-transversal-one-line-P} and $\widetilde{m}\leqslant m$, we deduce that
\begin{equation}
\label{equation:tilde-m-transversal-one-line-P}
\widetilde{m}<\Big(\frac{7}{5}+\frac{\epsilon_k}{2}\Big)\lambda-1<1.
\end{equation}

Denote by $\widehat{L}$ and $\widehat{\Omega}$ the proper transforms on $\widehat{S}$ of the divisors $\widetilde{L}$ and $\widetilde{\Omega}$, respectively.
Then the log pair
$$
\Big(\widehat{S},a\widehat{L}+\widehat{\Omega}+\big(a+b+m-1\big)\widehat{E}_1+\big(2a+b+m+\widetilde{m}-2\big)E_2\Big)
$$
is not log canonical at the point $O$.
Note that $2a+b+m+\widetilde{m}-2<1$ by \eqref{equation:mult-Q-transversal-one-line-P}.
Thus, using \eqref{equation:tilde-m-transversal-one-line-P} and arguing as in Subsection~\ref{subsection:1},
we see that $O\in\widehat{L}\cup\widehat{E_1}$.

If $O\in\widehat{E_1}$, then the log pair
$$
\Big(\widehat{S},\widehat{\Omega}+\big(a+b+m-1\big)\widehat{E}_1+\big(2a+b+m+\widetilde{m}-2\big)E_2\Big)
$$
is not log canonical at $O$. Applying Corollary~\ref{corollary:inversion-of-adjunction} again, we obtain $a+b+m+\widetilde{m}>2$,
so that \eqref{equation:a>-transversal-one-line-P} and \eqref{equation:mult-Q-transversal-one-line-P} give
$$
\Big(\frac{9}{5}+\epsilon_k\Big)\lambda\geqslant 2a+b+m+\widetilde{m}>2+a>3-\frac{\lambda}{2},
$$
which is impossible, since $\lambda<\frac{6}{5}$ and $\epsilon_k\to 0$ as $k\to \infty$.

We see that $O\in\widehat{L}$. Then the log pair
$$
\Big(\widehat{S},a\widehat{L}+\widehat{\Omega}+\big(2a+b+m+\widetilde{m}-2\big)E_2\Big)
$$
is not log canonical at the point $O$.
Now using \eqref{equation:mult-Q-transversal-one-line-P} and \eqref{equation:tilde-m-transversal-one-line-P}, we obtain
$$
\mathrm{mult}_O\Big(\widehat{\Omega}+\big(2a+b+m+\widetilde{m}-2\big)E_2\Big)=2a+b+m+2\widetilde{m}-2<\Big(\frac{12}{5}+\frac{3\epsilon_k}{2}\Big)\lambda-3<1,
$$
because $\lambda<\frac{6}{5}$ and $\epsilon_k\to 0$ as $k\to \infty$.
Thus, applying Lemma~\ref{lemma:Kewei-inequality}, we get
$$
L\cdot\Omega+2a+b-1=\widehat{L}\cdot\Big(\widehat{\Omega}+\big(2a+b+m+\widetilde{m}-2\big)E_2\Big)>2-a,
$$
which gives $L\cdot\Omega+3a+b>4$.
Using \eqref{equation:line-conic-transversal-relation}, we get $\lambda+4a>4+b\geqslant 4$,
so that \eqref{equation:a<5/9-transversal-one-line-P} implies that
$$
\Big(\frac{29}{9}-\varepsilon_k\Big)\lambda>4,
$$
which is impossible, since $\lambda<\frac{6}{5}$ and $\varepsilon_k\to 0$ as $k\to \infty$.

\subsection{Case 4}
\label{subsection:4}
We have $T_P=L+C$, where $L$ is a line, and $C$ is an irreducible conic that tangents $L$ at the point $P$.
We write
$$
\lambda D=aL+bC+\Omega,
$$
where $a$ and $b$ are nonnegative rational numbers, and $\Omega$ is an effective $\mathbb{Q}$-divisor whose support does not contain $L$ and $C$.
Let $m=\mathrm{mult}_P(\Omega)$. Then
\begin{equation}
\label{equation:mult>1-tangential-one-line-P}
a+b+m>1
\end{equation}
by Lemma \ref{lemma:not-lc-mult-large}.
Meanwhile, it follows from Lemma \ref{lemma:mult-of-a-line} that
\begin{equation}
\label{equation:a<5/9-tangential-one-line-P}
a\leqslant\Big(\frac{5}{9}+\varepsilon_k\Big)\lambda<1,
\end{equation}
where $\varepsilon_k$ is a small constant depending on $k$ such that $\varepsilon_k\to 0$ as $k\to \infty$.
And also, we have
\begin{equation}
\label{equation:line-conic-tangential-relation}
L\cdot\Omega=\lambda+a-2b.
\end{equation}

Denote by $\widetilde{L}$ and $\widetilde{C}$ the proper transforms on $\widetilde{S}$ of the curves $L$ and $C$, respectively.
We know that $Q=\widetilde{L}\cap\widetilde{C}$.
Denote by $\widetilde{\Omega}$ the proper transforms on $\widetilde{S}$ of the divisor $\Omega$.
Then the log pair
$$
\Big(\widetilde{S},a\widetilde{L}+b\widetilde{C}+\widetilde{\Omega}+\big(a+b+m-1\big)E_1\Big)
$$
is not log canonical at the point $Q$.
Since $a<1$ by \eqref{equation:a<5/9-tangential-one-line-P}, we may apply Corollary~\ref{corollary:inversion-of-adjunction}
to this log pair at $Q$ with respect to the curve $\widetilde{L}$.
This gives
$$
L\cdot\Omega>2-a-2b.
$$
Combining this with \eqref{equation:line-conic-tangential-relation}, we get $\lambda+2a>2$,
so that
\begin{equation}
\label{equation:a>-tangentail-one-line-P}
a>\frac{2-\lambda}{2}.
\end{equation}

Let $\widetilde{m}=\mathrm{mult}_Q(\widetilde{\Omega})$. Then Lemma~\ref{lemma:Q-in-L-C-tangential-blowup-one-line-P} gives
\begin{equation}
\label{equation:mult-Q-tangential-one-line-P}
2a+2b+m+\widetilde{m}=\lambda\cdot\mathrm{mult}_Q(\pi^*(D))\leqslant\Big(\frac{17}{9}+\epsilon_k\Big)\lambda.
\end{equation}
where $\epsilon_k$ is a small constant depending on $k$ such that $\epsilon_k\to 0$ as $k\to \infty$.
Thus, using \eqref{equation:a>-tangentail-one-line-P} and $\widetilde{m}\leqslant m$, we deduce that
\begin{equation}
\label{equation:tilde-m-tangential-one-line-P}
\widetilde{m}<\Big(\frac{13}{9}+\frac{\epsilon_k}{2}\Big)\lambda-1<1.
\end{equation}

Denote by $\widehat{L}$, $\widehat{C}$ and $\widehat{\Omega}$ the proper transforms on $\widehat{S}$ of the divisors $\widetilde{L}$, $\widetilde{C}$ and $\widetilde{\Omega}$, respectively.
Then the log pair
$$
\Big(\widehat{S},a\widehat{L}+b\widehat{C}+\widehat{\Omega}+\big(a+b+m-1\big)\widehat{E}_1+\big(2a+2b+m+\widetilde{m}-2\big)E_2\Big)
$$
is not log canonical at $O$.
Moreover, it follows from \eqref{equation:mult-Q-tangential-one-line-P} that $2a+2b+m+\widetilde{m}-2<1$.
Thus, using \eqref{equation:tilde-m-tangential-one-line-P} and arguing as in Subsection~\ref{subsection:1},
we see that $O\in\widehat{L}\cup\widehat{C}\cup\widehat{E_1}$.

If $O\in\widehat{E_1}$, then the log pair
$$
\Big(\widehat{S},\widehat{\Omega}+\big(a+b+m-1\big)\widehat{E}_1+\big(2a+2b+m+\widetilde{m}-2\big)E_2\Big)
$$
is not log canonical at $O$.
In this case, Corollary~\ref{corollary:inversion-of-adjunction} applied to this log pair (and the curve $E_2$) gives
$a+b+m+\widetilde{m}>2$, so that \eqref{equation:a>-tangentail-one-line-P} and \eqref{equation:mult-Q-transversal-one-line-P} give
$$
\Big(\frac{17}{9}+\epsilon_k\Big)\lambda\geqslant 2a+2b+m+\widetilde{m}>2+a+b>3-\frac{\lambda}{2},
$$
which is impossible, since $\lambda<\frac{6}{5}$ and $\epsilon_k\to 0$ as $k\to \infty$.

If $O\in \widehat{C}$, then the log pair
$$
\Big(\widehat{S},b\widehat{C}+\widehat{\Omega}+\big(2a+2b+m+\widetilde{m}-2\big)E_2\Big)
$$
is not log canonical at $O$.
In this case, if we apply Corollary~\ref{corollary:inversion-of-adjunction} to this log pair with respect to $E_2$, we get
$b+\widetilde{m}>1$, so that \eqref{equation:mult-Q-tangential-one-line-P} gives
$$
2a+b+m+1<\Big(\frac{17}{9}+\epsilon_k\Big)\lambda-1.
$$
Combining this with \eqref{equation:mult>1-tangential-one-line-P}), we see that $a<(\frac{17}{9}+\epsilon_k)\lambda-2$,
so that \eqref{equation:a>-tangentail-one-line-P} gives
$$
\Big(\frac{43}{18}+\epsilon_k\Big)\lambda>3,
$$
which is impossible, since $\lambda<\frac{6}{5}$ and $\epsilon_k\to 0$ as $k\to \infty$.

We see that $O\in\widehat{L}$. Then the log pair
$$
\Big(\widehat{S},a\widehat{L}+\widehat{\Omega}+\big(2a+2b+m+\widetilde{m}-2\big)E_2\Big)
$$
is not log canonical at the point $O$.
Now using \eqref{equation:mult-Q-tangential-one-line-P}, \eqref{equation:tilde-m-tangential-one-line-P} and $\lambda<\frac{6}{5}$,
we deduce that
$$
\mathrm{mult}_O\Big(\widehat{\Omega}+\big(2a+2b+m+\widetilde{m}-2\big)E_2\Big)=
2a+2b+m+2\widetilde{m}-2<\Big(\frac{10}{3}+\frac{3\epsilon_k}{2}\Big)\lambda-3<1.
$$
since $\lambda<\frac{6}{5}$ and $k\to \infty$.
Then we may apply Lemma~\ref{lemma:Kewei-inequality} to get
$$
L\cdot\Omega+2a+2b-2=\widehat{L}\cdot\Big(\widehat{\Omega}+\big(2a+2b+m+\widetilde{m}-2\big)E_2\Big)>2-a,
$$
which gives $L\cdot\Omega+3a+2b>4$.
Using \eqref{equation:line-conic-tangential-relation}, we see that $\lambda+4a>4$,
so that \eqref{equation:a<5/9-tangential-one-line-P} gives
$$
\Big(\frac{29}{9}-\varepsilon_k\Big)\lambda>4,
$$
which is impossible, since $\lambda<\frac{6}{5}$ and $\varepsilon_k\to 0$ as $k\to \infty$.

The proof of Theorem \ref{theorem:main} is complete.

\end{document}